\documentclass[10pt]{amsart}
\usepackage{amsmath}
\usepackage{paralist}
\usepackage{graphicx} 
\usepackage{epstopdf}
\usepackage[colorlinks=true]{hyperref}
\usepackage{float}
\usepackage{color,comment}
\usepackage{caption}
\usepackage{wrapfig}
\newcommand{\x}{\textbf{x}}

\newcommand{\bigo}[1]{O(#1)}

\newtheorem{theorem}{Theorem}[section]

\newtheorem{lemma}[theorem]{Lemma}
\newtheorem{proposition}{Proposition}
\newtheorem{conjecture}{Conjecture}

\theoremstyle{definition}

\newtheorem{remark}{Remark}

\title[ three species food chain model]
      { Biological control via ``ecological" damping: An approach that attenuates non-target effects}

\author[Beauregard, Black, Parshad, Quansah]{}


 \email{rparshad@clarkson.edu}
 \email{kblack@clarkson.edu}
 \email{quansaek@clarkson.edu}
 \email{beauregama@sfasu.edu}

\begin{document}
\maketitle

\centerline{ Rana D. Parshad, Kelly Black, and Emmanuel Quansah}
\medskip
{\footnotesize
 \centerline{Department of Mathematics,}
 \centerline{Clarkson University,}
   \centerline{ Potsdam, New York 13699, USA.}

 }

\medskip

\centerline{ Matthew Beauregard}
\medskip
{\footnotesize
 \centerline{Department of Mathematics $\&$ Statistics,}
 \centerline{Stephen F. Austin State University}
   \centerline{Nacogdoches, TX, 75962, USA.}

 }

\begin{abstract}
In this work we develop and analyze a mathematical model of
biological control to prevent or attenuate the explosive increase
of an invasive species population in a three-species food chain. We allow for finite time blow-up in the model
as a mathematical construct to mimic the explosive increase in population, enabling the species to reach ``disastrous"
levels, in a finite time. We next propose various controls to drive down the invasive population growth and, in certain cases, eliminate blow-up.
The controls avoid chemical treatments and/or natural enemy introduction,
 thus eliminating various non-target effects associated with such classical methods.
We refer to these new controls as ``ecological damping", as their inclusion dampens the invasive species population growth. Further, we improve prior results on the regularity and Turing instability of the three-species model that were derived in \cite{PK14}. Lastly, we confirm the existence of spatio-temporal chaos.\end{abstract}

 \section{Introduction} \label{1}
Population dynamics are a fundamental aspect of many biological processes. In this paper, we introduce and investigate a mathematical model for the population dynamics of an invasive species in a three-species food chain model. Exotic species are defined as any species, capable of propagating themselves into a nonnative environment. If the propagating species is able to establish a self sustained population in this new environment, it is formally defined as invasive. The survival and competitiveness of a species depends intrinsically on an individual's fitness and ability to assimilate limited resources. Often invasive species possess the ability to dominate a particular resource. This allows them to expand their range via out-competing other native species.
In the United States damages caused by invasive species to agriculture, forests, fisheries and businesses, have been estimated to be $\$$120 billion a year \cite{Pimentel05}. In the words of Daniel Simberloff:
``\emph{Invasive species are a greater threat to native biodiversity than pollution, harvest, and disease combined.}'' \cite{S00}
Therefore understanding and subsequently attenuating the spread of invasive species is an important and practical problem in spatial ecology and much work has been devoted to this issue \cite{A06, B07,C01,M00,089,S97}. More recently, the spread of natural and invasive species by nonrandom dispersal, say due to competitive pressures, is of great interest \cite{A12,L12}.  However, there has been less focus, on the actual eradication of an invasive species, once it has already invaded. This is perhaps a harder problem. In the words of Mark Lewis:
\emph{``Once the invasive species are well established, there is not a lot you can do to control them''}. \cite{E12}
It is needless to say however, that in many ecosystems around the world, invasion has already taken place! Some prominent examples in the US, are the invasion by the Burmese python in southern regions of the United States, with climatic factors supporting their spread to a third of the United States \cite{Rodda05}.
The sea lamprey and round goby have invaded the Great Lakes region in the northern United States and Canada \cite{19}.  These species have caused a severe decline in lake trout and other indigenous fish populations. Lastly the Zebra Mussel has invaded many US and Canadian waterways causing large scale losses to the hydropower industry \cite{N94}.

Another factor attributed to the increase of an invasive population, is that the environment may turn favorable for the invasive species in question, while becoming unfavorable for its competitors or natural enemies.  In such situations, the population of the invasive species may rapidly increase. This is defined as an \textit{outbreak} in population dynamics \cite{B87}. These rapid changes tend to destabilize an ecosystem and pose a threat to the natural environment. As an illustration, in the European Alps certain environmental conditions have enabled the population of the larch budmoth to become large enough that entire forests have become defoliated \cite{L79}.
In most ecological landscapes, due to exogenous factors, one always encounters an invasive species and an invaded species. If the density of the former, be it invasive, disease causing, an agricultural pest, defoliator or other, undergoes a rapid transition to a high level in population, the results can be catastrophic both for local and nonlocal populations.

Biological and chemical controls are an adopted strategy to limit invasive populations \cite{V96}.
Chemical controls are most often based on direct methods, via the use of pesticides \cite{V96}. Biological control comprises of essentially releasing natural enemies of the invasive species/pest against it. These can be in the form of predators, parasitoids, pathogens or combinations thereof \cite{V96}. There are many problems with these approaches.
For example, a local eradication effort was made by USGS through a mass scale poisoning of fish in order to prevent the asian carp (an invasive fish species) from entering the Chicago Sanitary and Ship Canal \cite{CT09}.  The hope was to protect the fishing interests of the region. However, among the tens of thousands of dead fish, \emph{biologists found only one asian carp}. Thus chemical control is not an exhaustive strategy. However, biological controls are also not without its share of problems. In fact, sometimes the introduced species might attack a variety of species, other than those it was released to control. This phenomena is referred to as a \emph{non-target effect} \cite{F00,L03}, and is common in natural enemies with a broad host range. For example, the cane toad was introduced in Australia in 1935 to control the cane beetle. However, the toad seemingly attacked everything else but its primary target \cite{18}! In addition, the toad is highly poisonous and therefore predators shy away from eating it.  This has enabled the toad population to grow virtually unchecked and is today considered one of Australia's worst invasive species \cite{CaneToad}. In studies of biological control in the United States estimate that when parasitoids are released as biological controls that $16\%$ of the introduced species will attack non-targets, in Canada these numbers are estimated as high as $37.5\%$ \cite{F00}. In practise it is quite difficult to accurately predict these numbers.
The current drawbacks make it clear that alternative controls are necessary.  Furthermore, that modeling of alternative controls is important to validate the effectiveness of a management strategy that hopes to avoid non-target effects. Such modeling is essential to access and predict the best controls to employ, so that the harmful population will decrease to low and manageable levels. This then gives us confidence to devise actual field trials. We should note that in practice actual eradication is rarely achieved.

Thus there are clear questions that motivate this research:
\begin{enumerate}
\item How does one define a ``high'' level for a population, and further, how well does an introduced control actually work, at various high levels?
\item Is it possible to design controls that avoid chemicals/pesticides/natural enemy introduction, and are still successful?
\end{enumerate}

This paper addresses these questions through the investigation of a mathematical model that:
\begin{enumerate}
\item Blows-up in finite time. Given a mathematical model for a nonlinear process, say through a partial differential equation (PDE), one says finite time blow up occurs if
\begin{equation*}
\lim_{t\rightarrow T^{\ast}<\infty}\| r \|_{X} \rightarrow \infty,
\end{equation*}
where $X$ is a certain function space with a norm $\|\cdot\|$, $r$ is the solution to the PDE in question, and $T^{\ast}$ is the blow up time. Therefore ``highest'' level is equated with blow up, and the population passes through every conceivable high level of population as it approaches infinity.
\item Incorporates certain controls that avoid chemicals/pesticides/natural enemy introduction.  The controls we examine are:

\begin{enumerate}
    \item The primary food source of the invasive species is protected through spatial refuges. The regions that offer protection are called \textit{prey refuges} and may be the result of human intervention or natural byproducts, such as improved camouflage.
   \item An overcrowding term is introduced to model the movement or dispersion from high concentrations of the invasive species. Densely populated regions have increased intraspecific competition and interference which cause an increase in the dispersal of the invasive species.  This is an improvement to current mathematical models and will be seen to be beneficial if used a control.
    \item We introduce role reversing mechanisms, where the role of the primary food source of the invasive species, and the prey of this food source, is switched in the open area (the area without refuge). This models situations where the topography provides competitive advantages to certain species.  It will be seen that this also is beneficial if used a control.  In effect, this uses the current ecosystem and by modifying the landscape a natural predator in the environment has an advantage in key areas. Hence, the invasive species population will adversely be effected. It can also be thought of as introducing a competitor of the invasive species, to compete with it for its prey.
    \end{enumerate}
\end{enumerate}

\begin{remark}
Note, none of the above rely on enemy release to predate on the invasive species, or a parasite or pathogen release to infect the invasive species. Thus potential non-target effects due to such release can be avoided.
\end{remark}

In the literature finite time blow up is also referred to as an explosive instability \cite{S98}. There is a rich history of blow up problems in PDE theory and its interpretations in physical phenomenon. For example, this feature is seen in models of thermal runaway, fracture and shock formation, and combustion processes. Thus blow up may be interpreted as the failure of certain constitutive materials leading to gradient catastrophe or fracture, it may be interpreted as an uncontrolled feedback loop such as in thermal runaway, leading to explosion. It might also be interpreted as a sudden change in physical quantities such as pressure or temperature such as during a shock or in the ignition process. The interested reader is referred to \cite{QS7,S98}. Blow up in population dynamics is usually interpreted as excessively high concentrations in small regions of space, such as seen in chemotaxis problems \cite{H09}. Our goal in the current manuscript is to bring yet another interpretation of blow up to population dynamics, that is one where we equate such an excessive concentration or ``blow up" of an invasive population with disaster for the ecosystem. Furthermore, it is to devise controls that avoid non-target effects and yet reduce the invasive population, before the critical blow up time.

In the following, the norms in the spaces $\mathbb{L}^{p}(\Omega )$, $\mathbb{L}^{\infty
}(\Omega )$ and $\mathbb{C}\left( \overline{\Omega }\right) $ are
respectively denoted by
\[
\left\Vert u\right\Vert _{p}^{p}\text{=}\frac{1}{\left\vert \Omega
\right\vert }\int\limits_{\Omega }\left\vert u(x)\right\vert ^{p}dx, ~~
\left\Vert u\right\Vert _{\infty }\text{=}\underset{x\in \Omega }{\max}%
\left\vert u(x)\right\vert .
\]
In addition, the constants $C$, $C_1$ and $C_2$ may change between subsequent lines, as the analysis permits, and even in the same line if so required.

The current manuscript is organised as follows. In section \ref{2} we formulate the spatially explicit model that we consider. In section \ref{3} we describe in detail the modeling of the control mechanisms that we propose, and term ``ecological damping". Here we make three key conjectures \ref{c1}, \ref{c2} and \ref{c3}, concerning our control mechanisms. Section \ref{4} is devoted to some analytical results given via lemma \ref{lem:wsol}, theorem \ref{thm:wsolimproved} and \ref{thm:gattr}. Section \ref{5} is where we explain our numerical approximations and test conjectures \ref{c1}, \ref{c2} and \ref{c3} numerically. In section \ref{6} we investigate spatio-temporal dynamics in the model. We investigate the effect of overcrowding on Turing patterns, and we also confirm spatio-temporal chaos in the model. Lastly we offer some concluding remarks and discuss future directions in section \ref{7}.

\section{Model Formulation}
\label{2}
A three species food chain model is developed, where the top predator, denoted as $r$, is invasive.
In our model, $r$ may blow up in finite time. Although populations cannot reach infinite values in finite time, they can grow rapidly \cite{B87}. 
The blow up event indicates that the invasive population has reached ``an extremely high'' and uncontrollable level. Naturally, that level occurs prior to the blow up time. Therefore, as $\|r\|$ approaches infinity in finite time, it passes through every conceivable ``high'' level. The blow up time, $T^{\ast}$, is viewed as the ``disaster'' time. We investigate mechanisms that attempt to lower and control the targeted population \emph{before time $T^{\ast}$}. This modeling approach has distinct advantages:

\begin{enumerate}
\item There is no ambiguity as to what is a disastrous high level of population.

\item There is a clear demarcation between when or if the disaster occurs.

\item The controls that are proposed do not rely on a direct attack on the invasive species, as is the traditional approach, rather we attempt to control the food source of $r$.  This will avoid possible nontarget effects.

\item The model provides a useful predictive tool that can be tuned and established through data, in various ecological settings. 

\item Mathematical models are advantageous in many situations due to their cost-effectiveness and versatility. Of course obtaining an analytical solution for a nonlinear model is virtually impossible, outside of special cases.  However, a numerical approximation can be developed to accurately investigate the role and effect our controls have on the blow up behavior. 
\end{enumerate}

Suppose an invasive species $r$ has invaded a certain habitat and it has become the top predator in a three species
food chain. Hence, $r$ predates on a middle predator $v$, which in turn predates on a prey $u$. A temporal model is given for the species interaction, namely
\begin{eqnarray}
\label{eq:x3o}
\frac{dr}{dt} &=&  cr^{2}-w_{3}\frac{r^{2}}{v+D_{3}}, \\
\label{eq:x2o}
\frac{dv}{dt} &=&  -a_{2}v+w_{1}\left(\frac{u v}{u + D_{1}}\right)-w_{2}\left(\frac{vr}{v+D_{2}}\right),\\
\label{eq:x1o}
\frac{du}{dt} &=&  a_{1}u-b_{2}u^{2}-w_{0}\left(\frac{u v }{u+D_{0}}\right).
\end{eqnarray}
The spatial dependence is included via diffusion,
\begin{eqnarray}
\label{eq:(1.3)}
\partial _{t}r &=& d_{3}\Delta r+cr^{2}-\omega _{3}\frac{r^{2}}{v+D_{3}} = h(u,v,r),\\
\label{eq:(1.2)}
\partial _{t}v &=& d_{2}\Delta v-a_{2}v+w_{1}\frac{uv}{u+D_{1}}-w_{2}\frac{vr}{v+D_{2}}=g(u,v,r),\\
\label{eq:(1.1)}
\partial _{t}u &=& d_{1}\Delta u+a_{1}u-b_{2}u^{2}-w_{0}\frac{uv}{u+D_{0}}=f(u,v,r),
\end{eqnarray}
\noindent defined on $\mathbb{R}^{+}\times \Omega$. Here $\Omega \subset \mathbb{R}^{N}$, $N=1,2$ and $\Delta$ is the one or two dimensional Laplacian operator.  We define $\x$ to be the spatial coordinate vector in one or two dimensions. The parameters $d_{1}$, $d_{2}$ and $d_{3}$ are positive diffusion coefficients.
Neumann boundary conditions are specified on the boundary.  The initial populations are given as
\[ u(0,\x)=u_{0}(\x),\;v(0,\x)=v_{0}(\x)\text{,\ }r(0,\x)=r_{0}(\x)\ \;\;\ \text{in } \Omega, \]
\noindent are assumed to be nonnegative and uniformly bounded on $\Omega$.

There are various parameters in the model: $a_1,~ a_2,~ b_2$, $w_0,~w_1,~w_2,$ $w_3,~ c,~ D_0,$ $D_1,~D_2,$ and $D_3$ are all positive constants. Their definitions are as follows: $a_1$ is the growth rate of prey $u$; $a_2$ measures the rate at which $v$ dies out when there is no $u$ to prey on and no $r$; $w_{i}$ is the maximum value that the per-capita rate can attain; $D_0$ and $D_1$ measure the level of protection provided by the environment to the prey;  $b_{2}$ is a measure of the competition among prey, $u$; $D_2$ is the value of $v$ at which its per capita removal rate becomes ${w_2}/2$;  $D_3$ represents the loss in $r$ due to the lack of its favorite food, $v$; $c$ describes the growth rate of $r$ via sexual reproduction.

These models offer rich dynamics and were originally proposed in \cite{U98,U97} in order to explain why chaos has rarely been observed in natural populations of three interacting species. The model stems from the Leslie-Gower formulation and considers interactions between a generalist top predator, specialist middle predator and prey. The study of these models have generated much research \cite{G05,N13,L02,P10,PK14,UIR00,U7,U5}. An interesting fact is if $c > \frac{w_3}{D_{3}}$ the spatially independent and spatially dependent models are easily seen to blow up in finite time \cite{PK13}.  The spatially dependent model offers further rich dynamics, in particular the possibility of Turing instabilities and non Turing patterns \cite{N13, PK14}.
Nevertheless, to avoid blow up it appears that one must restrict $c < K\frac{w_3}{D_{3}}$, where $K < 1$. This was established for the spatially independent model in \cite{AA02} and is offered here for convenience:

\begin{theorem}
\label{thm:aziz}
Consider the model \eqref{eq:x3o}-\eqref{eq:x1o}. Under the assumption that
\begin{equation}
\label{eq:ac1}
c < \left(  \frac{w_0 b_2 D_3}{w_1\left( a_1 + \frac{a_1^2}{4a_2}\right) + w_0 b_2 D_3}\right) \frac{w_3}{D_3},
\end{equation}
\noindent all non-negative solutions (i.e. solutions initiating in $\mathbb{R}^{3}_{+}$) of \eqref{eq:x3o}-\eqref{eq:x1o} are uniformly bounded forward in time and they eventually enter an attracting set $\mathcal{A}$.
\end{theorem}

\begin{remark}
We have recently shown the above result to be \emph{false} in the ODE and PDE cases. That is, \eqref{eq:(1.3)}-\eqref{eq:(1.1)} may blow up in finite time, even under \eqref{eq:ac1} provided the initial data is large enough \cite{PKK15}.  It is clear that even if \eqref{eq:ac1} is maintained $r$ can blow up.  This becomes more evident if we consider the coefficient $c-\frac{w_{3}}{v+D_{3}}$ on $r^2$ and the fact that if the fecundity of $r$ is large compared to $\frac{w_{3}}{v+D_{3}}$ then blow up will occur.  This may happen in situations where:
\begin{enumerate}
\item There is an abundance of $v$,
\item $r$ possesses certain abilities to out compete other species, and harvest enough $v$,
\item The environment has turned favorable for $r$ and unfavorable for its natural enemies or competitors. Thus it can harvest $v$ unchecked.
\end{enumerate}
\end{remark}

\begin{remark}
If no intervention is made we can envision $r$ growing to disastrous levels with adequate initial resources. The blow up time $T^{*}$ is viewed as the point of disaster in an ecosystem.  Thus one is interested in controlling $r$ via the use of biological controls, before the critical time $T^{*}$.
\end{remark}

\begin{remark}
 These observations motivate an interesting question. Assume that both \eqref{eq:x3o}-\eqref{eq:x1o} and its spatially explicit form \eqref{eq:(1.3)}-\eqref{eq:(1.1)} blow up in finite time for $c < \frac{w_{3}}{D_{3}}$ for a given initial condition. Can we modify \eqref{eq:(1.3)}-\eqref{eq:(1.1)}, via introducing certain controls, so that now \emph{there is no blow up}, for the same initial condition?
\end{remark}

\section{Delay and Removal of blow up via ``Ecological Damping"}
\label{3}
It is clear that controlling the population of an invasive specie is advantageous, and most often necessary. However the avenues for which this is possible, whilst avoiding non-target effects, is not clear. Here, we propose new controls that may delay or even remove blow-up in the invasive population.  We refer to these controls as ``ecological damping", akin to damping forces such as friction in physical systems, that add stability to a system. The crux of our idea is to use prey refuges, in conjunction with role reversal and overcrowding effects. There is a vast literature on prey refuge, spatial refuges as well as role reversal in the literature. The interested reader is referred to \cite{C12,F06,K04}. However, to the best of our knowledge these have not been proposed as control mechanisms for invasive species.


\subsection{Prey Refuge: Modified Model I}

\begin{figure}[h]
\begin{center}
  {\includegraphics[scale=.40]{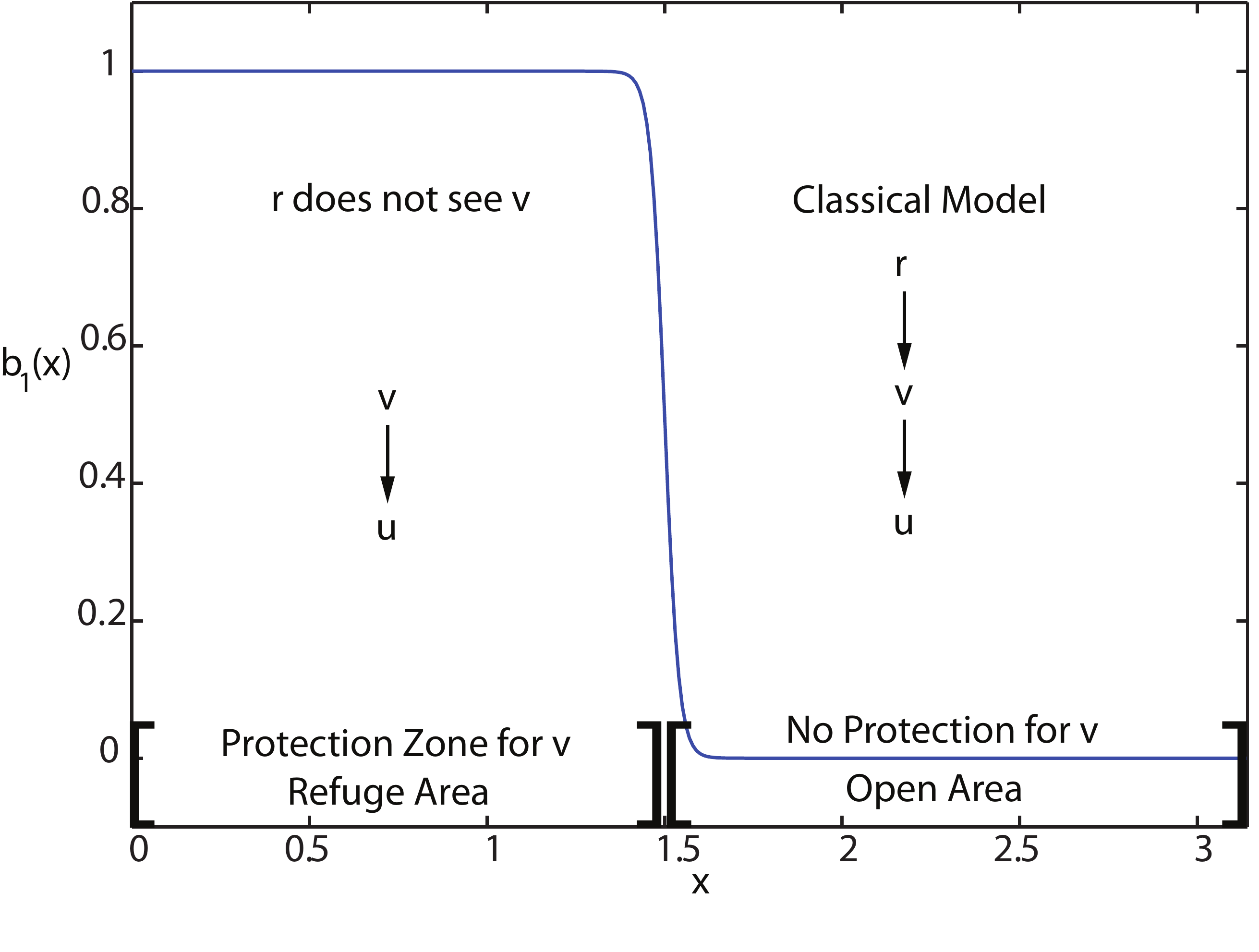}}
  \caption{\small An example of a prey refuge function $b_{1}(x) = \frac{1}{2}\left(1-\tanh(\frac{x-a}{\epsilon})\right)$ in the one dimensional case.  This function decreases monotonically throughout its domain. The range is $0 \leq b_{1}(x) \leq 1$. If $b_1(x)=1$ then the prey $v$ is protected while if $b_1(x)=0$ then $v$ is unprotected. }
  \label{Fig:refugeplot}
\end{center}
 \end{figure}

We consider modeling a prey refuge.  Consider the continuous function $b_1(\x)$. The region where $b_1(\x)=1$ or sufficiently close to one is defined as a \textit{prey refuge domain} or \textit{patch}.  We call the region where $b_1(\x)=0$ or sufficiently close to zero an open area, this is the region where $r$ can predate on $v$. In Figure \ref{Fig:refugeplot} we see a sharp gradient between the prey refuge domain and the open area.

The inclusion of a prey refuge influences the equations for $r$ and $v$, namely,
\begin{eqnarray}
\label{eq:x3pa}
r_t &=& d_3 \Delta r + cr^{2}-w_{3}\frac{r^{2}}{(1-b_1))v+D_{3}} \\
\label{eq:x2pa}
v_t &=& d_2 \Delta v -a_{2}v+w_{1}\left(\frac{u v}{u + D_{1}}\right)-w_{2}(1-b_1)\left(\frac{vr}{v+D_{2}}\right),
\end{eqnarray}
posed on a bounded domain in one or two dimensions.  Neumann boundary conditions are specified. The equation for $u$ may be found in \eqref{eq:(1.1)}.

The introduction of the prey refuge creates regions where it is impossible for $r$ to predate on $v$.  Notice, that if the entire spatial domain is considered a prey refuge, that is, $b_1(\x)=1$ $\forall \x \in \Omega$, then the coefficient of $r^2$ will depend on the sign of $c < \frac{w_{3}}{D_{3}}$. If this is negative, then the invasive species dies off.  Likewise, in the absence of a prey refuge, that is, $b_1(\x)=0$ $\forall \x\in \Omega$ the equations for $r$ and $v$ collapse to our previous ones, that is, \eqref{eq:(1.3)} and \eqref{eq:(1.2)}, respectively.  In such a case, it is know that blow up may still occur for sufficiently large enough data even if $c < \frac{w_{3}}{D_{3}}$. This is because the coefficient of $r^2$ may still be positive, as one always has $\frac{w_{3}}{v+D_{3}} < c < \frac{w_{3}}{D_{3}}$.
The introduction of the refuge \emph{forces} the coefficient of $r^2$ to change sign between the refuge and the open area. In the literature such problems are referred to as indefinite parabolic problems. The word \emph{indefinite} refers to the sign of the coefficient being indefinite. Although there is a vast literature on such problems for single species models \cite{G04,Q08}, there is far less work for systems. There is also a large amount of literature for such switching mechanisms incorporated to understand competitive systems \cite{GM05, G06}, particularly in the vein of human economic progress.

This indefinite parabolic problem motivates a collection of questions:  If blow up occurs for particular parameters in the case $b_1(\x)=0$, will a prey refuge prevent blow up? How does this depend on the geometry of the refuge?  What is the critical size or shape of a refuge that prevents blow up from occurring? Does this depend on the size of the initial condition? In the situation that blow up still persists, how is the blow up time affected? Are these results influenced by multiple refuges?

In either case, we make the conjecture:

\begin{conjecture}
\label{c1}
Consider the three species food chain model \eqref{eq:(1.3)}-\eqref{eq:(1.1)}, a set of parameters with $c < \frac{w_3}{D_3}$,
and an initial condition $(u_{0},v_{0},r_{0})$
such that $r$ which is the solution to \eqref{eq:(1.3)}, blows up in finite time, that is
\begin{equation*}
\mathop{\lim}_{t \rightarrow T^{*} < \infty} \int_{\Omega}r(\x,t)d\x \rightarrow \infty ,
 \end{equation*}
\noindent there exists a patch $\Omega_{1} \subset \Omega$, s.t for any single patch of measure greater than or equal to $|\Omega_{1}|$, the modified model \eqref{eq:x3pa}-\eqref{eq:x2pa}, with the same parameter set and initial condition, has globally existing solutions. In particular
the solution $r$ to \eqref{eq:x3pa}, does not blow up in finite time.
\end{conjecture}

\begin{proof}
A proof of this conjecture can be established for certain special cases in one dimesion. Assume \eqref{eq:(1.3)}-\eqref{eq:(1.1)} blows up at time $t=T^{*}$, for a certain parameter set and an initial condition $r_0(x)$, such that $r_{0}(\pi)=0$. We consider now introducing a refuge at the right end, starting at some positive $a<\pi$. Now for $x \in [a,\pi]$ the density of $v$ decreases, and the coefficient of $r^2$ is $\left( c - \frac{w_{3}}{D_{3}}\right) < 0$, hence $r$ is restricted on the right hand side of the domain. We assume that this is equivalent to introducing a Dirichlet boundary condition somewhere in the closed interval $[a,\pi]$. Thus the modified model is equivalent to
\begin{equation}
\label{eq:x3pann}
r_t = d_3 \Delta r + cr^{2}-w_{3}\frac{r^{2}}{v+D_{3}}, \ \mbox{on} \ [0,b], \ \mbox{where} \ a < b < \pi,
\end{equation}
\noindent and $r_{x}(0,t)=0$ and $r(b,t)=0$. Here we assume the initial condition satisfies $r_{0}(b)=0$.
Let us now compare \eqref{eq:x3pann} to
\begin{equation}
\label{eq:x3pann1}
r_t= d_3 \Delta r + cr^{2}, \ \mbox{on} \ [0,b], \ \mbox{where} \ a < b < \pi.
\end{equation}
\noindent The solution of the above is a supersolution to \eqref{eq:x3pann}. However we know that there exists small data solutions to \eqref{eq:x3pann1}. This is easily seen via following the methods in \cite{zheng04}. Basically without loss of generality we may assume $c=d_{3}=1$. We then multiply \eqref{eq:x3pann1} by $r_t$ and integrate by parts in $(0,b)$ to obtain
\[
\frac{d}{dt}\left( \frac{1}{2}\|r_x\|^{2}_{2}-\frac{1}{3}\|r\|^3_3 \right)+\|r_t\|^{2}_{2}=0,
\]
thus
\[
\frac{d}{dt}\left\{ \|r_x\|^{2}_{2}-\frac{2}{3}\|r\|^3_3 \right\}\le 0,
\]
We now define the functional
\[
E(t)=\|r_x(t)\|^{2}_{2}-\frac{2}{3}\|r(t)\|^3_3,
\]
\noindent Since $E'(t)\le 0$ we obtain
\[
E(t)\le E(0)=\|\partial_x r_0\|^{2}_{2}-\frac{2}{3}\|r_0\|^3_3.
\]
\noindent Now we multiply \eqref{eq:x3pann1} by $r$ and integrate integrate by parts in $(0,b)$ to obtain
\begin{equation}  \label{eq:rp-1}
\frac{1}{2}\frac{d}{dt} \|r\|^{2}_{2} + \|r_x\|^2 - \|r\|^3_3 = 0.
\end{equation}
This yields
\begin{equation}  \label{rp-2}
\frac{1}{2}\frac{d}{dt} \|r\|^{2}_{2} + E - \frac{1}{3}\|r\|^3_3 = 0.
\end{equation}
\noindent So If $E(0)<0,$ then $E(t)<E(0)<0$  and so
\begin{equation}  \label{rp-3}
\frac{1}{2}\frac{d}{dt} \|r\|^{2}_{2} \ge \frac{1}{3}\|r\|^3_3.
\end{equation}
\noindent This essentially yields
\begin{equation}  \label{rp-4}
\frac{d}{dt} \|r\|^{2}_{2} \ge \frac{2}{3}\left(\|r\|^{2}_{2}\right)^{\frac{3}{2}}.
\end{equation}
Setting $Y(t):=\|r(t)\|^{2}_{2}$. We  derive the following differential inequality $Y'\ge\frac{2}{3}Y^{3/2}$.
and so
\[
\|r(t)\|^{2}_{2}\ge \frac{3\|r_0\|^{2}_{2}}{3-t\|r_0\|^{2}_{2}}.
\]
which means the solutions exist for $t\in[0,T^*)$, where
\begin{equation}  \label{bu-time}
T^*=\frac{3}{\|r_0\|^{2}_{2}},\quad r_0\not=0.
\end{equation}
\noindent and then the solution $r$ blows up at time $T^{*}$. This means that if $E(0)\ge 0$, then the initial data is actually small enough to ensure globally existing solutions \cite{zheng04}. What is required is
\begin{equation}  \label{global-id}
\|\partial_x r_0\|^{2}_{2}\ge\frac{2}{3}\|r_0\|^3_3.
\end{equation}
\noindent This criteria can always be obtained for large enough refuge, that is, for $b$ small enough. Since the boundary terms cancel, the norm here is equivalent to the $H^{1}_{0}(0,b)$ norm, hence by Sobolev embedding we have
\begin{equation}  \label{eq:global-id}
\int^{b}_{0}|\partial_x r_0|^{2}dx\ge C \left(\int^{b}_{0}|r_0|^{3}dx\right)^{\frac{2}{3}}
\end{equation}
 Now since $\int^{b}_{0}|r_0|^{3}dx << 1$, for $b$ chosen small enough we obtain
\begin{equation}  \label{eq:global-id2}
 C \left(\int^{b}_{0}|r_0|^{3}dx\right)^{\frac{2}{3}} > \frac{2}{3}\left(\int^{b}_{0}|r_0|^{3}dx\right)
\end{equation}
Thus combining \eqref{eq:global-id} and \eqref{eq:global-id2} we obtain
\begin{equation}  \label{eq:global-id3}
\int^{b}_{0}|\partial_x r_0|^{2}dx\ge C \left(\int^{b}_{0}|r_0|^{3}dx\right)^{\frac{2}{3}} >\frac{2}{3}\left(\int^{b}_{0}|r_0|^{3}dx\right).
\end{equation}
This proves a particular case of conjecture \ref{c1}.
\end{proof}

\subsection{The Overcrowding Effect: Modified Model II}

In high population density areas a species should have greater dispersal in order to better assimilate available resources and avoid crowding effects such as increased intraspecific competition. These effects can be modeled via an overcrowding term, that has not been included in mathematical models of biological control.  Consider the improved mathematical model that includes an overcrowding effect of the invasive species $r$, namely,
\begin{eqnarray}
\label{eq:x3pc}
r_t &=& d_3 r_{xx} + cr^{2}-w_{3}\frac{r^{2}}{v+D_{3}} + d_4(r^2)_{xx},\\
\label{eq:x2pam2}
v_t &=& d_2 v_{xx} -a_{2}v+w_{1}\left(\frac{u v}{u + D_{1}}\right)-w_{2}\left(\frac{vr}{v+D_{2}}\right),
\end{eqnarray}
\noindent with initial and boundary conditions as before. Again, the equation for $u$ remains unchanged.  The addition of $d_4(r^2)_{xx}\equiv\frac{d_4}{2} (rr_x)_{x}$ represents a severe penalty on local overcrowding. This is interpreted as movement from high towards low concentrations of $r$, directly proportional to $r$. Hence, $r$ attempts to avoid overcrowding and disperses toward lower concentrations. Such models have been under intense investigation recently and are referred to as cross-diffusion and self-diffusion systems \cite{S97}. The mathematical analysis of such models is notoriously difficult \cite{ J06, JC04}. We limit ourselves to the one-dimensional case in the forthcoming analysis and its numerical approximations. Therefore the one dimensional Laplacian is considered in \eqref{eq:(1.1)}.

The presence of blow up is not affected if we maintain Neumann boundary conditions at the boundaries.  This can be seen in a straightforward manner.  Let us assume the classical model, that is \eqref{eq:(1.3)}-\eqref{eq:(1.1)}, blows up.  Therefore, without loss of generality, there exists a $\delta$ such that $\left(c-\frac{w_{3}}{v+D_{3}} \right)> \delta > 0$. This implies that if we set $H(t)= \int_{\Omega}r(x,t)dx$, then,
\[
\frac{d}{dt}H(t) \geq \frac{\delta}{\sqrt{|\Omega|}}  H(t)^2,
\]
leading to the blow up of $H(t)$. Now, consider the integration over $\Omega$ of \eqref{eq:x3pc}.  Since the overcrowding term integrates to zero then blow up still persists. However, a combination of a prey refuge with the overcrowding effect may prevent blow up.  Further, we expect it takes a smaller refuge to accomplish the removal of blow up. This is precisely the following conjecture:
\begin{conjecture}
\label{c2}
Consider the three species food chain model \eqref{eq:(1.3)}-\eqref{eq:(1.1)}, a set of parameters, such that $c < \frac{w_3}{D_3}$,
and an initial condition $(u_{0},v_{0},r_{0})$
such that $r$ which is the solution to \eqref{eq:(1.3)}, blows up in finite time, that is
\begin{equation*}
\mathop{\lim}_{t \rightarrow T^{*} < \infty} \int_{\Omega}r(\x,t)d\x \rightarrow \infty ,~~ T^{*} < \infty,
 \end{equation*}
\noindent there exists a patch $\Omega_{2} \subset \Omega$, and an overcrowding coefficient $d_4$, s.t for any single patch of measure greater than or equal to $|\Omega_{2}|$,
the modified model \eqref{eq:x3pc}-\eqref{eq:x2pam2}, with the same parameter set and initial condition, has globally existing solutions. In particular
The solution $r$ to \eqref{eq:x3pc}, does not blow up in finite time. Furthermore, $\Omega_{2} \subset \Omega_{1}$, where $\Omega_{1}$ is the patch found in conjecture \eqref{c1}.
\end{conjecture}

This is not difficult to see, as multiplying through by $r$ and integrating by parts yields
\[
\frac{1}{2}\frac{d}{dt} \|r\|^{2}_{2} + E - \frac{1}{3}\|r\|^3_3 + \int_{\Omega}r|r_{x}|^{2}dx= 0.
\]
\noindent So If $E(0)<0,$ then $E(t)<E(0)<0$  and so
\begin{equation}  \label{eq:rp-3n}
\frac{1}{2}\frac{d}{dt} \|r\|^{2}_{2} + \int_{\Omega}r|r_{x}|^{2}dx \ge \frac{1}{3}\|r\|^3_3 > \frac{2}{3}\left(||r||^{2}_{2}\right)^{\frac{3}{2}}.
\end{equation}

Thus even if one has negative energy, $E(t) \leq E(0) < 0$ so that \eqref{eq:x3pa} blows-up, due to the presence of the positive $\int_{\Omega}r|r_{x}|^{2}dx$ term in \eqref{eq:rp-3n}, \eqref{eq:x3pc} will not blow up for small data, thus yielding a global solution. Thus there are initial data for which \eqref{eq:x3pa} blow ups but \eqref{eq:x3pc} does not, and so a smaller refuge would work in this case. The method follows by mimicing the steps in \eqref{eq:rp-1}-\eqref{eq:global-id3}, with the additional $\int_{\Omega}r|r_{x}|^{2}dx$ term.

\subsection{Refuge, Overcrowding, and Role Reversal: Modified Model III}

The prey refuge and overcrowding are included in the one-dimensional mathematical model. We also include a role-reversal of $u$ within the protection zone of the refuge.  This models the scenario in which two species may prey on each other in various regions where it is advantageous.  Hence, the role-reversal of $u$ will compete with the invasive species $r$.  Figure \ref{Fig:refugeplotrolereversal} depicts the scenario of a one dimensional prey refuge for which outside the protection zone both $u$ and $r$ may predate on $v$. Hence, the model we propose of this scenario is given below,
 \begin{figure}[h]
\centering{
  {\includegraphics[scale=.4]{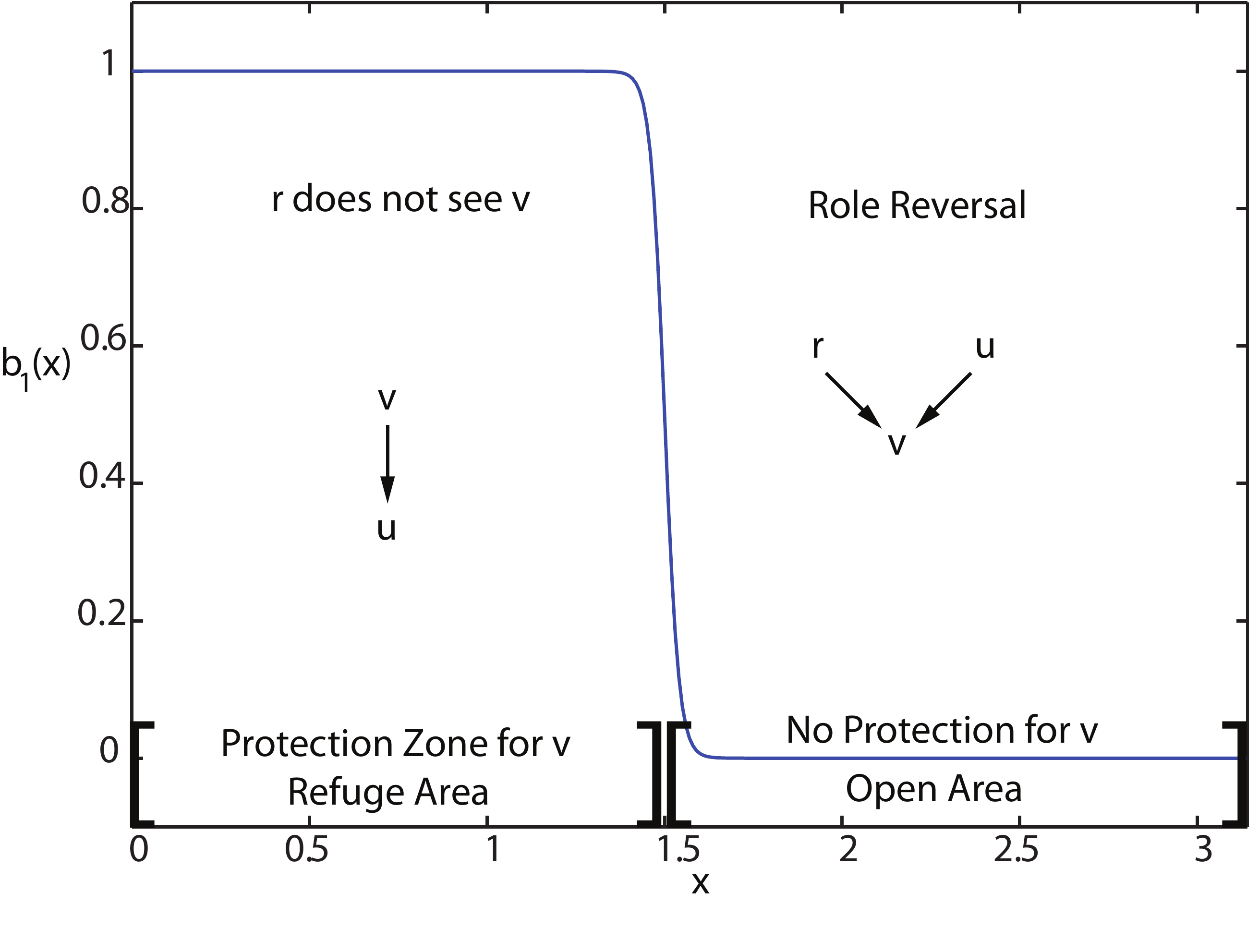}}
  }
  \caption{\small This shows a plot of the refuge function $b_{1}(x) = \frac{1}{2}\left(1-\tanh(\frac{x-a}{\epsilon})\right)$.  In certain regions $u$ has a disadvantage over $v$, while outside of this $u$ is able to prey on $v$.  The invasive species effect on $v$ is influenced by the prey refuge.
}
\label{Fig:refugeplotrolereversal}
 \end{figure}
\begin{eqnarray}
\label{eq:xn3pcm3n}
 r_t  &=&  d_3 r_{xx} + (c -\frac{w_3}{(1-b_1(x))v + D_3}) r^2 + d_{4}(r^{2})_{xx},\\
 v_t  &=&  d_2 v_{xx} - a_2 v  + b_1(x)w_1\frac{uv}{u+D_1} \nonumber \\
 \label{eq:xn3pcm2n}
 && - (1-b_1(x)) \left(w_4 \frac{vu}{v+D_2}   + w_2\frac{vr}{v+D_4}\right),\\
\label{eq:xn3pcm1n}
 u_t  &=&  d_1 u_{xx} +  a_{1}u - b_{2}u^2 + (1-b_1(x)) w_5 \frac{vu}{v+D_0} - b_1(x)w_1\frac{uv}{u+D_3} .
\end{eqnarray}
In light of this model we propose Conjecture \ref{c3}:

\begin{conjecture}
\label{c3}
Consider the three species food chain model \eqref{eq:(1.3)}-\eqref{eq:(1.1)}, a set of parameters such that $c < \frac{w_3}{D_3}$,
and an initial condition $(u_{0},v_{0},r_{0})$
such that $r$ which is the solution to \eqref{eq:(1.3)}, blows up in finite time, that is
\begin{equation*}
\mathop{\lim}_{t \rightarrow T^{*} < \infty} \int_{\Omega}r(\x,t)d\x \rightarrow \infty ,~ T^{*} < \infty,
\end{equation*}
\noindent there exists a patch $\Omega_{3} \subset \Omega $  such that
the modified model \eqref{eq:xn3pcm3n}-\eqref{eq:xn3pcm1n}, with the same parameter set and initial condition, has globally existing solutions. In particular, the solution $r$ to \eqref{eq:xn3pcm3n} does not blow up in finite time. Furthermore, $\Omega_{3} \subset \Omega_{2}$, where $\Omega_{2}$ is the
patch found in conjecture \eqref{c2}.
\end{conjecture}

All three conjectures are tested numerically in Section \ref{sec:numresults}. 

\section{Formal estimates and regularity}
\label{4}
\subsection{Preliminaries}

The nonnegativity of the solutions of \eqref{eq:(1.3)}-\eqref{eq:(1.1)} is preserved by application of standard results on invariant regions \cite{S83}. This is due to to the reaction terms being quasi-positive, that is,
\[
f\left( 0,v,r\right) \geq 0,\ g\left( u,0,r\right) \geq 0,\ h\left(
u,v,0\right) \geq 0\ ,\ \ \text{\ for all }u,v,r \geq 0.
\]

Since the reaction terms are continuously differentiable on $\mathbb{R}%
^{+3}$, then for any initial data in $\mathbb{C}\left( \overline{\Omega }%
\right) $ or $\mathbb{L}^{p}(\Omega ),\;p\in \left( 1,+\infty \right) $, it
is easy to directly check their Lipschitz continuity on bounded subsets of
the domain of a fractional power of the operator $I_{3}\left(
d_{1},d_{2},d_{3}\right) ^{t}\Delta $, where $I_{3}$ the three dimensional
identity matrix, $\Delta $ is the Laplacian operator and $\left( {}\right)
^{t}$ denotes the transposition. Under these assumptions, the following
local existence result is well known (see \cite{F64,H84,P83,R84,S83}).

\begin{proposition}
\label{prop:ls}
The system \eqref{eq:(1.3)}-\eqref{eq:(1.1)} admits a unique, classical solution $(u,v,r)$ on $%
[0,T_{\max }[\times \Omega $. If $T_{\max }<\infty $ then
\[
\underset{t\nearrow T_{\max }}{\lim }\left\{ \left\Vert u(t,.)\right\Vert
_{\infty }+\left\Vert v(t,.)\right\Vert _{\infty }+\left\Vert
r(t,.)\right\Vert _{\infty }\right\} =\infty .
\]
\end{proposition}

\begin{lemma}[Uniform Gronwall Lemma]
\label{lem:gronwall}
Let $\beta,~ \zeta$, and $h$ be nonnegative functions in $L^{1}_{loc}[0,\infty;\mathbb{R})$. Assume that $\beta$
is absolutely continuous on $(0,\infty)$ and the following differential inequality is satisfied:
\[
\frac{d \beta}{dt} \leq \zeta \beta + h, \ \mbox{for} \ t>0.
\]
If there exists a finite time $t_{1} > 0$ and some $q > 0$ such that
\[
\int^{t+q}_{t}\zeta(\tau) d\tau \leq A,  \ \int^{t+q}_{t}\beta(\tau) d\tau \leq B, \ \mbox{and} \ \int^{t+q}_{t} h(\tau) d\tau \leq C,
\]
\noindent for any $t > t_{1}$, where $A, B$, and $C$ are some positive constants, then
\[
\beta(t) \leq \left(\frac{B}{q}+C\right)e^{A}, \ \mbox{for \ any} \ t > t_{1}+q.
\]
\end{lemma}

\subsection{Improvement of Regularity of the Classical Model}

\subsubsection{Improvement of the Global Existence Condition}

In this section we improve the global existence conditions that has been derived recently in \cite{PK14}. Here, we provide the previous result, Theorem \ref{thm:wsol}.
\begin{theorem}
\label{thm:wsol}
Consider the three-species food-chain model described by  \eqref{eq:(1.3)}-\eqref{eq:(1.1)}. For any initial data $(u_0,v_0,r_0) \in L^{2}(\Omega)$, such that $||v_0||_{\infty} \leq \frac{w_3}{c} -D_3$, and parameters $(c,w_3,D_3)$ such that $||v||_{\infty} \leq \frac{w_3}{c} -D_3$, there exists a global classical solution $(u,v,r)$ to the system.
\end{theorem}

We now give an improved result for the ODE case, which is easily modified for PDE case. In essence, given any initial data (however large), there is global solution, for $a_2$ appropriately large. This is summarised in following lemma,

\begin{lemma}
\label{lem:wsol}
Consider the three-species food-chain model described by  \eqref{eq:x3o}-\eqref{eq:x1o}. As long as $c < \frac{w_3}{D_3}$, given any initial data $(u_0,v_0,r_0)$, however large, there exists a global solution $(u,v,r)$ to the system as long as the parameter $a_2$ is s.t
\[
a_2 \geq w_1 + c|r_{0}|\ln\left( \frac{|v_{0}|}{w_{3}/c - D_{3}}\right),
\]
\end{lemma}
\begin{proof}
Consider the following subsystem
\begin{eqnarray}
\label{eq:x2ne}
\frac{dv}{dt} &=&  -a_{2}v+w_{1}v,\\
\label{eq:x3on}
\frac{dr}{dt} &=&  cr^{2}.
\end{eqnarray}
\noindent The $v$ here grows faster than the $v$ determined by \eqref{eq:x2o}. Similarly, the $r$ here blows up faster than $r$ in \eqref{eq:x3o}. In order to drive $v$ down below $w_{3}/c - D_{3}$, we use the exact solution to \eqref{eq:x2ne}, that is,
$$v = e^{-(a_2 - w_{1})t}v_{0}$$
\noindent Assuming that $a_{2} > w_1$, then for $v$ to be below $w_{3}/c - D_{3}$ implies that
\[
e^{-(a_2 - w_{1})t}v_{0} \leq w_{3}/c - D_{3}
\]
\noindent Thus when $t > T^{*} = \frac{1}{a_{2} - w_{1}} \ln\left( \frac{|v_{0}|}{w_{3}/c - D_{3}}\right)$, we shall have
\[
e^{-(a_2 - w_{1})t}v_{0} = v \leq w_{3}/c - D_{3}.
\]
\noindent Now \eqref{eq:x3on} blows up at time $T^{**} = \frac{1}{c|r_{0}|}$.  So if we choose $a_{2}$ appropriately, we can make $T^{*} < T^{**} = \frac{1}{c|r_{0}|}$. This is done by choosing
\[
\frac{1}{a_{2} - w_{1}} \ln\left( \frac{|v_{0}|}{w_{3}/c - D_{3}}\right) = T^{*} < T^{**} = \frac{1}{c|r_{0}|}
\]
\noindent Thus if $a_{2}$ is s.t $a_2 \geq w_1 + c|r_{0}|\ln\left( \frac{|v_{0}|}{w_{3}/c - D_{3}}\right)$, then $v \leq w_{3}/c - D_{3}$, before $r$ blows up. Therefore, if we consider the $r$ determined by \eqref{eq:x3o} then this will certainly not blown up by time $t=T^{*}$. At this point however, $c - \frac{w_{3}}{v + D_{3}}$ is negative, so $r$ in \eqref{eq:x3o}, cannot blow up after $T^{*}$, and will decay. In this latter case the global attractor is a very simple, $(u,v,r)=(u^*,v^*,0)$. Hence, we have an extinction state for $r$.

In short, we see that for any initial condition that is specified, there exists an $a_2$ (depending on the initial condition), s.t \eqref{eq:x3o}-\eqref{eq:x1o} has a global solution.
\end{proof}


\subsubsection{Uniform $L^2(\Omega)$ and $H^{1}(\Omega)$ Estimates}

We recap the following estimate from \cite{PK13}
\[
||u||^{2}_{2} \leq  e^{- \frac{b_1}{C_1}t}||u(0)||^{2}_{2}+ \frac{CC_1}{b_{1}}.
\]
\noindent Therefore, there exists time $t_{1}$ given explicitly by
\[
t_{1}=\max\left\{0,C_1\frac{\ln(||u(0)||^{2}_{2})}{b_{1}}\right\},
\]
\noindent such that, for all $t \geq t_{1}$, the following estimate holds uniformly:
\[
||u||^{2}_{2} \leq  1 + \frac{CC_{1}}{b_{1}}.
\]
\noindent Here $t_1$ is the compactification time of $||u||_2$.

We also recall the following local in time integral estimate for $\nabla u$, from \cite{PK13}
\[
\int^{t+1}_{t}||\nabla u||^{2}_{2}ds \leq \frac{1}{2d_1}||u(t)||^{2}_{2} + \frac{1}{d_1}\int^{t+1}_{t}Cds \leq \frac{1}{2d_1}\left(1 + \frac{CC_1}{b_{1}}\right)+\frac{C}{d_1}, \ \mbox{for} \ t > t_{1}.
\]
Now we move to the $L^2$ estimate for the $v$ component. This cannot be obtained directly. Thus we use the grouping method again by multiplying \eqref{eq:(1.1)} by $w_1$ and \eqref{eq:(1.2)} by $w_0$, adding the two together, and setting $w=w_1u+w_0v$, we obtain:
\begin{equation}
\label{eq:x1nw}
w_t= d_2 \Delta w + (d_1-d_2)w_1 \Delta u + w_1a_{1}u-w_1b_{1}u^{2}-w_0a_2v-w_0w_2\left(\frac{vr}{v+D_2}\right).
\end{equation}
\noindent We add a convenient zero, $-a_2w_1u + a_2w_1u$, to \eqref{eq:x1nw} to obtain
\[
w_t= d_2 \Delta w + (d_1-d_2)w_1 \Delta u + w_1a_{1}u-w_1b_{1}u^{2}-a_2w + a_2w_1u-w_0w_2\left(\frac{vr}{v+D_2}\right).
\]
\noindent By multiplying \eqref{eq:x1nw} by $w$, integrating by parts over $\Omega$,
and keeping in mind the positivity of the solutions, we find that
\begin{eqnarray*}
&&\frac{1}{2}\frac{d}{dt}||w||^{2}_{2} + d_{2}|| \nabla w||^{2}_{2} + a_2||w||^{2}_{2} \nonumber \\
&& \leq  w_1(a_{1}+a_{2})\int_{\Omega}u w dx + \int_{\Omega}(d_2-d_1)w_1 \nabla u \cdot \nabla w dx .
\end{eqnarray*}
\noindent This yields
\begin{eqnarray}
\label{eq:x11**}
&&\frac{d}{dt}||w||^{2}_{2} + 2\min(a_2,d_{2})\left(|| \nabla w||^{2}_{2} + ||w||^{2}_{2}\right) \nonumber \\
&& \ \ \ \leq  2w_1(a_{1}+a_{2})\int_{\Omega}u w dx + 2(d_2-d_1)w_1 \int_{\Omega}\nabla u \cdot \nabla w dx \nonumber \\
&& \ \ \ \leq \min(a_2,d_{2})|| w||^{2}_{2} + \frac{(2w_1(a_{1}+a_{2}))^2}{2\min(a_2,d_{2})}|| u||^{2}_{2} + \min(a_2,d_{2})||\nabla w||^{2}_{2}  \nonumber \\
&& \ \ \ \ \  + \frac{(2(d_2-d_1)w_1)^2}{2\min(a_2,d_{2})}||\nabla u||^{2}_{2}, \nonumber
\end{eqnarray}
\noindent such that
\begin{eqnarray}
\label{eq:x11ny}
&& \frac{d}{dt}||w||^{2}_{2} + \min(a_2,d_{2})\left(||w||^{2}_{2} + || \nabla w||^{2}_{2}\right) \nonumber \\
&& ~~~~\leq \frac{(2w_1(a_{1}+a_{2}))^2}{2\min(a_2,d_{2})}|| u||^{2}_{2} + \frac{(2(d_2-d_1)w_1)^2}{2\min(a_2,d_{2})}||\nabla u||^{2}_{2}. \nonumber
\end{eqnarray}
Now via estimates in \cite{PK13} we obtain
\[ ||w||^{2}_{2} \leq C, \]
\noindent for $t > t_{3}$ where $t_{3}$ is explicitly derived in \cite{PK13}, $t_{3} = \max \left\{0, \ln\left(\frac{w_{1}\|u_{0}\|^{2}_{2} + w_{0}\|v_{0}\|^{2}_{2} }{C}\right) \right\}$ which implies that
\[
||v||^{2}_{2} \leq C, \ \mbox{for} \ t > t_{3}.
\]
\noindent Notice by multiplying \eqref{eq:(1.2)} by $v$ and integrating by parts one obtains,
\[
\frac{1}{2}\frac{d}{dt}||v||^{2}_{2} + d_{2}|| \nabla v||^{2}_{2} + a_2||v||^{2}_{2}
 \leq  w_1 ||v||^{2}_{2}
\]
\noindent Integrating the above on $[t_{3},t_{3}+1]$ yields
\[
\int^{t_{3}+1}_{t_{3}}|| \nabla v||^{2}_{2} ds \leq \int^{t_{3}+1}_{t_{3}}w_1 ||v||^{2}_{2}ds +\frac{1}{2} ||v(t_{3})||^{2}_{2} \leq (w_{1}+1)C \leq C
\]
\noindent Thus via the mean value theorem for integrals we obtain a time $t^{*}_{3} \in [t_{3},t_{3}+1]$ s.t
\[
|| \nabla v(t^{*}_{3})||^{2}_{2} ds  \leq C
\]
In addition, we recap the uniform $H^1$ estimates found in \cite{PK13},
\begin{equation}
\label{eq:x1-h11}
||\nabla u||^{2}_{2} \leq C,
\end{equation}
\noindent for $ t > T_{1}$, the $H^1$ compactification time of the
solution. We next need to derive $H^{1}$ estimates for the $v$
component. This is tricky, since when we multiply  \eqref{eq:(1.2)} by
$-\Delta v$, integrate by parts, and then apply Young's inequality
with epsilon we obtain
\[
\frac{1}{2}\frac{d}{dt}\|\nabla v\|^{2}_{2} + a_2\| \nabla v\|^{2}_{2} \leq w_{1}\| v\|^{2}_{2} + w_{2}\| r\|^{2}_{2}
\]
However, there is no global estimate on $r$, as it may possibly blow
up, but we know that there is always a local solution on $[0,\frac{1}{2}T^{*}]$, where $T^{*}$ is the blow up time of $r_t = d_3 \Delta r + cr^{2}$. Furthermore, on this time interval the solution is classical. Thus we make local in time estimates on this interval, using standard local in time estimates on $||r||_{\infty}$ from \cite{QS7}. Using Gronwalls inequality via integration on the time interval $[t^{*}_{3},t]$ yields,
\begin{eqnarray*}
\|\nabla v\|^{2}_{2} &\leq& e^{-2a_{2}t}\| \nabla v(t^{*}_{3})\|^{2}_{2} + \frac{1}{a_{2}}\left(w_{1}C+w_{2}C_{1}\frac{c||r_{0}||_{\infty}}{d_{3}} \right)\nonumber \\
&\leq&  1 + \frac{1}{a_{2}}\left(w_{1}C+w_{2}C_{1}\frac{c||r_{0}||_{\infty}}{d_{3}} \right)
\end{eqnarray*}
\noindent for $t > T_{4} = \max \left\{ 0,  \frac{ln(\|\nabla v(t^{*}_{3})\|^{2}_{2} )}{a_{2}} , t^{*}_{3}\right\}$.

Note, Lemma \ref{lem:wsol} requires us to manipulate $a_{2}$, in order to prove large data global existence. From a practical point of view this is not always easy as manipulating the death rate directly may not be condusive to realistic control strategies. We next provide the following improved result that is valid even if $||v_0||_{\infty} \geq \frac{w_3}{c} - D_3$, without manipulating $a_{2}$.

\begin{theorem}
\label{thm:wsolimproved}
Consider the three-species food-chain model described by  \eqref{eq:(1.3)}-\eqref{eq:(1.1)}. For certain initial data $(u_0,v_0,r_0) \in L^{2}(\Omega)$,  there exists a global classical solution $(u,v,r)$ to the system, even if $||v_0||_{\infty} \geq \frac{w_3}{c} -D_3$, as long as the parameters are such that $||v||_{\infty} \leq \frac{w_3}{c} - D_3$, by the $H^{2}$ absorption time of $v$ ( denoted $T_{6}$), If $T_{6}$ satisfies $T_{6} \leq T^{*}$, where $T^{*}$ is the blow up time of the following PDE
\[
\partial _{t}r-d_{3}\Delta r=cr^{2}.
\]
with the same initial and boundary conditions as \eqref{eq:(1.3)}
\end{theorem}

In order to show that blow up can be avoided as stated above, even in $\mathbb{R}^2$ our $H^1$ estimates via \eqref{eq:x1-h11} must be improved. This is because the embedding of $H^1 \hookrightarrow L^{\infty}$ is lost in $\mathbb{R}^2$. Thus, if we are to control the $L^{\infty}$ norm in $\mathbb{R}^2$ we need $H^2$ control. This is achieved in the next subsection.

\subsubsection{Uniform $H^2(\Omega)$ Estimates}

 We will estimate the $H^2$ norms via the following procedure, we rewrite \eqref{eq:(1.1)} as
\[
u_t- d_1 \Delta u =  a_{1}u-b_{2}u^{2}-w_{0}\left(\frac{u v }{u+D_{0}}\right).
\]
\noindent We square both sides of the equation and integrate by parts over $\Omega$ to obtain
\begin{eqnarray}
\label{eq:x1pq}
&&||u_t||^{2}_{2} + d_{1}||\Delta u||^{2}_{2}  + \frac{d}{dt}||\nabla u||^{2}_{2}  \nonumber \\
&=& \left\| \left( a_{1}u-b_{2}u^{2}-w_{0}\left(\frac{u v }{u+D_{0}}\right) \right) \right\|_2^2  \nonumber \\
&\leq& C\left((a_{1})^2||u||^{2}_{2}+(w_{0})^2||v||^{2}_{2}\right)  + (b_{2})^2||u||^{4}_{4} \leq C. \nonumber
\end{eqnarray}
\noindent This result follows by the embedding of $H^{1} \hookrightarrow L^4 \hookrightarrow L^2$.  Therefore, we obtain
\[
||u_t||^{2}_{2} + d_{1}||\Delta u||^{2}_{2}  + \frac{d}{dt}||\nabla u||^{2}_{2}  \leq C.
\]
\noindent Due to classical local in time regularity results, see \ref{prop:ls}, the solutions are $C^{1}(0,T;C^{2}(\Omega))$), thus $\frac{d}{dt}||\nabla u||^{2}_{2} $ cannot escape to $- \infty$, and is bounded below, upto the existence time. Thus we obtain
\[
||u_t||^{2}_{2} \leq C_1(t),~~~ ||\Delta u||^{2}_{2}    \leq C_2(t).
\]
These constants may depend on $t$ since even though $\frac{d}{dt}||\nabla u||^{2}_{2} $ cannot escape to $- \infty$ in finite time the case of it decaying like $-t$ is not precluded.
We now make uniform in time estimates of the higher order terms. Integrating the estimates of \eqref{eq:x1pq} in the time interval $[T,T+1]$, for $T > T_{5}=\max(T_{1},T_{4})$, where $T_{5}$ is the maximum of the $H^1$ compactification times of $u,v$ yields
\[
\int^{T+1}_{T}||u_t||^{2}_{2} dt \leq C_1,~  \int^{T+1}_{T}||\Delta u||^{2}_{2}dt \leq C_2.
\]
\noindent Similarly we adopt the same procedure for $v$ to obtain
\[
\int^{T+1}_{T}||v_t||^{2}_{2}dt \leq C_1,~  \int^{T+1}_{T}||\Delta v||^{2}_{2}dt \leq C_2.
\]
\noindent However, the above estimates for $v$ are made on the time interval $[0,\frac{T^{*}}{2}]$.

Next, consider the gradient of \eqref{eq:(1.1)}. Following the same technique as in deriving \eqref{eq:x1pq} we obtain for the left hand side
\begin{eqnarray}
\label{eq:gut1}
&& ||\nabla u_t||^{2}_{2} + d_{1}||\nabla (\Delta u)||^{2}_{2}  + \frac{d}{dt}||\Delta u||^{2}_{2} + \int_{\partial \Omega}\Delta u \nabla u_{t} \cdot \textbf{n} dS \nonumber \\
&&=||\nabla u_t||^{2}_{2} + d_{1}||\nabla (\Delta u)||^{2}_{2}  + \frac{d}{dt}||\Delta u||^{2}_{2} + \int_{\partial \Omega}\Delta u \frac{\partial}{\partial t}(\nabla u \cdot \textbf{n} )dS \nonumber \\
&&= ||\nabla u_t||^{2}_{2} + d_{1}||\nabla (\Delta u)||^{2}_{2}  + \frac{d}{dt}||\Delta u||^{2}_{2}
\end{eqnarray}
\noindent This follows via the boundary condition. Thus we have
\begin{eqnarray}
\label{eq:x1pq}
&& ||\nabla u_t||^{2}_{2} + d_{1}||\nabla (\Delta u)||^{2}_{2}  + \frac{d}{dt}||\Delta u||^{2}_{2}  \nonumber \\
&& = \left(\nabla ( a_{1}u-b_{2}u^{2}-w_{0}\left(\frac{u v }{u+D_{0}}\right) )\right)^2 \nonumber \\
&& \leq C(|| u||^{4}_{4} + || v||^{4}_{4} + || \Delta u||^{2}_{2} + || \Delta v||^{2}_{2})
\end{eqnarray}
\noindent This follows via Young's inequality with epsilon, as well as the embedding of $H^{2}(\Omega)\hookrightarrow W^{1,4}(\Omega)$. This implies that
\[
\frac{d}{dt}||\Delta u||^{2}_{2}  \leq C || \Delta u||^{2}_{2} + C(|| u||^{4}_{4} + || v||^{4}_{4} + || \Delta v||^{2}_{2})
\]
\noindent We now use the uniform Gronwall lemma with
\[ \beta(t) = ||\Delta u||^{2}_{2}, \ \zeta(t)=C , \ h(t)= C(|| u||^{4}_{4} + || v||^{4}_{4} + || \Delta v||^{2}_{2}),~~ q=1 \]
\noindent to obtain
\[
||\Delta u||_{2} \leq C,  \ ||\Delta v||_{2} \leq C, \ \mbox{for} \ t > T_{6} = T_{5} + 1
\]
The estimate for $v$ is derived similarly, but again we assume we are on $[0,\frac{T^{*}}{2}]$.
\noindent Thus, via elliptic regularity,
\begin{equation}
\label{eq:h2ea}
|| u||_{H^{2}} \leq C,  \ ||v||_{H^{2}} \leq C, \ \mbox{for} \ t > T_{6}
\end{equation}
\noindent Since $H^2 \hookrightarrow L^{\infty}$ in $\mathbb{R}^2$ and $\mathbb{R}^{3}$, the following estimate is valid in $\mathbb{R}^2$ and $\mathbb{R}^{3}$
\begin{equation}
\label{eq:liea}
||u||_{\infty} \leq C,  \ ||v||_{\infty} \leq C, \ \mbox{for} \ t > T_{6}
\end{equation}
\noindent Thus even if the data is such that $||v_{0}||_{\infty} >  \frac{w_{3}}{c} - D_{3}$, if the parameters and the data $(r_{0},v_{0},u_{0})$ are such that $\frac{w_{3}}{c} - D_{3} < C_{1}C$, (for the $C$ in \eqref{eq:liea}, where $C_{1}$ is the embedding constant for $H^{2}(\Omega) \hookrightarrow L^{\infty}(\Omega)$) and
\begin{eqnarray*}
T_{6}(r_{0},v_{0},u_{0}) &<& 1+ \max \left\{   0,  T_{1},  \frac{\ln(\|\nabla v(t^{*}_{3})\|^{2}_{2} )}{a_{2}} ,  \ln\left(\frac{w_{1}\|u_{0}\|^{2}_{2} + w_{0}\|v_{0}\|^{2}_{2} }{C}\right)  +1  \right\}  \nonumber \\
&<& \frac{d_{3}}{c||r_{0}||_{\infty}} < T^{*},
\end{eqnarray*}
\noindent then we obtain
\[
 ||v||_{\infty} \leq \frac{w_{3}}{c} - D_{3} < C, \ \mbox{for} \ t > T_{6}.
\]
\noindent Since $T_{6} < T^{*}$ the solution to $\partial _{t}r-d_{3}\Delta r=cr^{2},$ has not blown up yet, and since this is a supersolution to $r$ solving \eqref{eq:(1.3)}, this $r$ has definitely not blown up. However, from this point on the coefficient of $r^2$ in \eqref{eq:(1.3)} is negative, as $\left(c-\frac{w _{3}}{v+D_{3}}\right) < \left(c-\frac{w _{3}}{||v||_{\infty}+D_{3}}\right) < 0$, hence $r$ solving \eqref{eq:(1.3)} can never blow up, from this point on. This proves Theorem \ref{thm:wsolimproved}.

\subsection{Long Time Dynamics}
\subsubsection{Gradient Estimates for the Time Derivatives}

Now we assume there are globally existing solutions, and we aim to investigate the regularity of the omega limit set. Hence we assume we are working with initial data and parameters for which $r$ is globally bounded.
To begin, we obtain useful integral in time estimates. In particular, we integrate \eqref{eq:x1pq} in time from $[T,T+1]$ where $T > T_{6}$, the $H^{2}$ absorbing time, to obtain
\begin{eqnarray}
&&\int^{T+1}_{T}||\nabla u_t||^{2}_{2}dt \leq C(||u||^{4}_{4} + || v||^{4}_{4} + || \Delta u||^{2}_{2} + || \Delta v||^{2}_{2}) + ||\Delta u(T)||^{2}_{2} \leq C. \nonumber
\end{eqnarray}
\noindent Similarly, an estimate for $\nabla v_{t}$ can be established, namely
\[
\int^{T+1}_{T}||\nabla v_t||^{2}_{2}dt  \leq C.
\]

We now develop higher order estimates for the time derivatives. First, consider the partial derivative w.r.t $t$ of equations \eqref{eq:(1.1)}, multipling by $-\Delta u_{t}$, and integrate over $\Omega$ we obtain
\begin{eqnarray*}
\frac{d}{dt}||\nabla u_t||^{2}_{2} + d_{1}||\Delta u_t||^{2}_{2}  &=& a_{1}||\nabla u_t||^{2}_{2} + b_{2}\int_{\Omega} u (u_t) \Delta u_{t} dx + w_{0} \int \frac{u}{u+D_{0}}v_{t}\Delta u_{t}dx \nonumber \\
&&~~ + w_{0} \int_{\Omega}v (u_{t}) \Delta u_{t} \frac{D_{0}}{(u+D_{0})^{2}}dx.
\end{eqnarray*}
\noindent Using Holder's inequality, Young's inequality with epsilon, and our earlier estimates yield
\begin{eqnarray*}
\frac{d}{dt}||\nabla u_t||^{2}_{2} + d_{1}||\Delta u_t||^{2}_{2} &\leq& a_{1}||\nabla u_t||^{2}_{2} + C||u||^{2}_{\infty}||u_{t}||^{2}_{2}+\frac{d_{1}}{4}||\Delta u_{t}||^{2}_{2}+ C_{1}||v_{t}||^{2}_{2}   \\
&&~~+ \frac{d_{1}}{4}||\Delta u_{t}||^{2}_{2}  + C_{2}||v||_{\infty} ||u_{t}||^{2}_{2}+\frac{d_{1}}{4}||\Delta u_{t}||^{2}_{2}.
\end{eqnarray*}
\noindent Thus we obtain
\begin{equation*}
\frac{d}{dt}||\nabla u_t||^{2}_{2}   \leq a_{1}||\nabla u_t||^{2}_{2} + C||u||^{2}_{\infty}||u_{t}||^{2}_{2}+ C_{1}||v_{t}||^{2}_{2}   + C_{2}||v||_{\infty} ||u_{t}||^{2}_{2}.
\end{equation*}
\noindent The application of the uniform Gronwall Lemma with
\begin{eqnarray*}
&&\beta(t) = ||\nabla u_t||^{2}_{2},~~ \zeta(t)=a_{1} , \\
&&h(t)= C||u||_{\infty}||u_{t}||^{2}_{2}+ C_{1}||v_{t}||^{2}_{2}   + C_{2}||v||_{\infty} ||u_{t}||^{2}_{2},~~ q=1
\end{eqnarray*}
\noindent gives us the following uniform bound
\begin{equation}
\label{eq:uh1-t}
||\nabla u_t||^{2}_{2}   \leq C ,~ t > T_{7} = T_{6} + 1
\end{equation}
Similar methods applied to the equation for $v$ and $r$ yield
\[
||\nabla v_t||^{2}_{2}   \leq C \ , t > T_{7} ,~ ||\nabla r_t||^{2}_{2}   \leq C , \ t > T_{7}.
\]
\subsubsection{Regularity of the omega limit set}
We now show that the asymptotic state $(0,v,u)$ under certain parameter restrictions is an attractor with more regularity than was derived previously. Essentially, we consider $(r_{0},v_{0},u_{0}) \in L^{2}(\Omega)$ for which we have globally existing solutions, that is the maximal existence time for the solutions $T = \infty$. Then we consider the omega limit set for such solutions.
\begin{equation}
\label{eq:om1}
\omega(r_{0},v_{0},u_{0}) = \overline{\bigcup_{t\geq 0}\left\{ (r(s),v(s),u(s)) : s\geq t  \right\} }
\end{equation}
\noindent We state the following result.

\begin{theorem}
\label{thm:gattr}
Consider the three-species food-chain model described by \eqref{eq:(1.3)}-\eqref{eq:(1.1)}. For initial conditions $(r_{0},v_{0},u_{0})$, such that there is a globally existing solution, there exists a parameter $a_{2}$, (depending on the initial condition), for which the omega limit set $\omega(r_{0},v_{0},u_{0}) = (0,v,u)$. Furthermore this is an attractor with $H^2(\Omega)$ regularity.
\end{theorem}
\begin{remark}
Thus for suitably chosen $a_{2}$, the attractor for such solutions is the extinction state of $r$.
\end{remark}
\begin{proof}
We have shown that the system is well posed, under certain parameter, and data restrictions, via theorem \ref{thm:wsolimproved}. Thus, there exists a well-defined semigroup $\left\{S(t)\right\}_{t \geq 0}:H \rightarrow H$. Also, via Lemma \ref{lem:wsol} we can find an $a_{2}$ such that $r$ decays to zero. The estimates via \eqref{eq:h2ea} establish the existence of bounded absorbing sets in $H^{2}$.
Now we rewrite the equation for a sequence $u_{n}$ as follows
\[
d_{1}  \Delta u_{n} = \partial_{t}u_{n} - a_{1}u_{n}+b_{2}u_{n}^{2}+w_{0}\left(\frac{u_{n} v_{n} }{u_{n}+D_{0}}\right),
\]
We aim to show the convergence of both the right hand side strongly in $L^2$. Due to the uniform bounds via \eqref{eq:x1-h11}, \eqref{eq:uh1-t}, \eqref{eq:h2ea} and the embedding of $H^1(\Omega) \hookrightarrow L^4(\Omega)  \hookrightarrow L^2(\Omega)$ and $H^2(\Omega) \hookrightarrow L^{\infty}(\Omega)$  we obtain a subsequence $u_{n_{j}}$ stil labeled $u_{n}$ s.t
\begin{eqnarray*}
 \partial_{t}u_{n}    &\xrightarrow{L^2}& \partial_{t}u^{*}, \\
 a_{1}u_{n}   &\xrightarrow{L^2}& a_{1}u^{*}, \\
 b_{2}u_{n}^2 &\xrightarrow{L^2}& b_{2}(u^{*})^{2}
\end{eqnarray*}
\noindent as $\int_{\Omega}|(u_{n}-u^*)(u_{n}+u^{*})|^{2}dx \leq ||u_{n}+u^{*}||_{\infty}||u_{n}-u^{*}||^{2}_{2} \rightarrow 0$, as $ u_{n} \xrightarrow{L^2} u^{*}$. Similarly,
\[
w_{0}\left(\frac{u_{n} v_{n} }{u_{n}+D_{0}}\right) \xrightarrow{L^2} w_{0}\left(\frac{u^{*} v^{*} }{u^{*}+D_{0}}\right)
\]
as $ v_{n} \xrightarrow{L^2} v^{*}$.  Thus, given a sequence $\left\{u_{n}(0)\right\}^{\infty}_{n=1}$  that is bounded in $L^{2}(\Omega)$, we know that, for $t > T_{7}$,
\[
 S(t_{n})(\Delta u_{n}(0)) \rightarrow \Delta u^{*} \ \mbox{in} \ L^{2}(\Omega).
\]
\noindent This yields the precompactness in $H^2(\Omega)$, hence the closure in \eqref{eq:om1} is in $H^2(\Omega)$. A similar analysis can be done for the $v$ and $r$ components. Hence the theorem is established.
\end{proof}

\section{Numerical Approximation}
\label{5}
In this section a numerical approximation for the one and two dimensional models are developed in order to numerically demonstrate our earlier conjectures \eqref{c1}, \eqref{c2}, and \eqref{c3} in addition to exploring the rich dynamics the mathematical models exhibit. A one-dimensional spectral-collocation method is developed to approximate the one dimensional equations \eqref{eq:xn3pcm3n}-\eqref{eq:xn3pcm1n}  which include overcrowding, refuges, and role-reversal. We then offer a second order finite difference approximation of the two-dimensional equations \eqref{eq:x3pa}, \eqref{eq:x2pa}, and \eqref{eq:(1.1)} that investigate the effect of a refuge.  Overcrowding and role-reversal effects are not investigated in the latter numerical approximation.  The two dimensional approximation employs a Peaceman-Rachford operator splitting and techniques developed by one of the authors in \cite{BeauregardSheng2013}.  This approach takes advantage of the sparsity and structure of the underlying matrices and the known computational efficiency of operator splitting methods.  Without loss of generality the domain is scaled and translated to $(-1,1)$ and $(-1,1)\times(-1,1)$ in one and two-dimensions, respectively.

%
\subsection{Approximation in One Dimension}

We develop a spectral-collocation approximation of \eqref{eq:xn3pcm3n}-\eqref{eq:xn3pcm1n}.   These equations can be written compactly as
\[ \textbf{w}_t = \mathcal{L}\textbf{w},\]
where $\textbf{w}=(u,v,r)^{\top}$ and $\mathcal{L}$ is the differential operator that includes the Laplacian operator and reactive terms.

The spatial approximation is constructed from a Chebychev collocation
scheme \cite{doi:10.1137/0916073,timeDependentSpectralMethods,doi:10.1137/0916006}.
The spatial approximation is constructed as a linear combination of
the interpolating splines on the Gauss-Lobatto quadrature. The
resulting system is then integrated in time using an implicit scheme, in particular a second order Adams-Moulton method. The Chebychev collocation approximation allows for a high order spatial approximation \cite{timeDependentSpectralMethods}. This offers the ability to capture fine resolution details with a relatively smaller
number of degrees of freedom.  However, there is a downside.  The resulting
matrices are full and ill conditioned \cite{doi:10.1137/0916006}. This is problematic for the inversion in the resulting linear solve.  Nevertheless, the spectrum
associated with the second derivative operator is real, negative, and
grows in magnitude like $O(N^4)$ \cite{doi:10.1137/0720063,timeDependentSpectralMethods}.


The collocation scheme is constructed on the
Gauss-Lobatto abscissa with respect to the Chebychev weight,
\begin{eqnarray}
  \label{eqn:numerics:abscissa}
  x_j & = & \cos\left(\frac{j\pi}{N}\right),
\end{eqnarray}
where $j=0,\ldots,N$. An approximation is then constructed as a linear
combination of the Lagrange interpolates on the abscissa,
\begin{eqnarray}
  \label{eqn:numerics:approximationSum}
  \textbf{w}_N(x,t) & = & \sum^N_{j=0} \textbf{a}_j(t) \phi_j(x),
\end{eqnarray}
where $\textbf{w}_N(x,t)$ is an approximation to the unknown populations and $\textbf{a}_j(t)$ is a $3 \times 1$ vector of fourier coefficients. 
The basis functions are defined by
\begin{eqnarray*}
  \phi_j(x) & = & \frac{(-1)^{N+j+1}(1-x^2) T_N'(x)}{c_j N^2 (x-x_j)}, \\
  c_j & = & \left\{
    \begin{array}{l@{\hspace{3em}}l}
      2 & \mathrm{if~}j=0,N, \\
      1 & \mathrm{otherwise.}
    \end{array}
    \right.,
\end{eqnarray*}
and each $\phi_j(x)$ is a polynomial of degree
$N$ \cite{timeDependentSpectralMethods}. The basis functions also
satisfy $\phi_j(x_i)=\delta_{ji}$. This is used to develop an approximation to our differential operator and a system of equations is constructed via a discrete inner product,
\begin{eqnarray*}
  <\textbf{w}_t,\phi_m>_N & = & <P_N \mathcal{L} \textbf{w}_N, \phi_m>_N,
\end{eqnarray*}
where $P_N$ is the projection operator onto the space of polynomials
of degree $N$. The discrete inner product is given below, and it is based on an inner product  defined by
\begin{eqnarray}
  \label{eqn:numerics:innerProduct}
  <u,v>_w & = & \int^1_{-1} \frac{u(x) v(x)}{\sqrt{1-x^2}} ~ dx.
\end{eqnarray}
The integral can be approximated as a sum over the Gauss-Lobatto
quadrature for the Chebychev weight and is exact for polynomials up
to degree $2N-1$,
\begin{eqnarray*}
  \int^1_{-1} \frac{p_{2N-1}(x)}{\sqrt{1-x^2}} ~ dx & = &
  \sum^N_{i=0} p_{2N-1}(x_k) w_k.
\end{eqnarray*}
The resulting norm results in an equivalent norm for polynomials
up to degree $2N$ \cite{SobolevCanutoQuarteroni}. The abscissa, $x_i$,
are the same as those defined in equation
\eqref{eqn:numerics:abscissa}, and the $w_i$ are the quadrature
weights. The resulting numerical approximation is then constructed
using this finite inner product
\begin{eqnarray*}
  <f,g>_N & = & \sum^N_{i=0} f(x_i)g(x_i) w_i.
\end{eqnarray*}
The resulting approximation is an extension of that given in \eqref{eqn:numerics:innerProduct} and is constructed using a Galerkin
approach,
\begin{eqnarray*}
    <\textbf{w}_t,\phi_m>_N & = & <P_N \mathcal{L} \textbf{w}_N, \phi_m>_N, \\
    \sum^N_{i=0} \frac{\partial}{\partial t} \textbf{w}_N(x_i,t) \phi_m(x_i) w_i & = &  \sum^N_{i=0} P_N \mathcal{L} \textbf{w}_N, \phi_m(x_i) w_i.
\end{eqnarray*}
The basis functions, $\phi_m(x)$, are Lagrange interpolants, and after
substitution of the definition in equation
(\ref{eqn:numerics:approximationSum}) the result can be greatly
simplified,
\begin{eqnarray*}
  \textbf{a}'_m(t) & = & P_N \mathcal{L} \textbf{w}_N \bigg|_{x=x_m},
\end{eqnarray*}
where $m=0,\ldots,N$.

The Adams-Moulton method used to advance the initial value problem in time necessitates solving a nonlinear system of equations at each time step.  This is achieved through a Newton-Raphson method.  The ensuing linearized problem in the Newton-Raphson method is solved through a restarted GMRES \cite{Morgan}.

\subsection{Approximation in Two Dimensions}

A second order approximation of \eqref{eq:x3pa}, \eqref{eq:x2pa}, and \eqref{eq:(1.1)} is developed.  While the spectral-Galerkin appoximation of the previous section may be extended to two dimensions, the computational cost of the nonlinear solve is troublesome due to the sensitivity of the matrices and the lack of their sparsity.  Here, the approximation is still of high order, while the underlying matrices are block tridiagonal or tridiagonal with diagonal blocks.  This enables direct inversion techniques, in particular the Thomas algorithm.

The governing equations are written compactly as
\begin{eqnarray*}
\textbf{w}_t &=& D \Delta \textbf{w} + \textbf{f},
\end{eqnarray*}
where $\textbf{w}=(u,v,r)^{\top}$, $\textbf{f}=(f(\textbf{w},\textbf{x},t),~g(\textbf{w},\textbf{x},t),~h(\textbf{w},\textbf{x},t))^{\top}$ are the reactive terms that depend on space, time, and the species $u,v,$ and $r$, $\textbf{x}=(x,y)$, $\Delta \textbf{w}$ is taken component-wise, and $D$ is a diagonal matrix with nonnegative entries $d_1,~d_2,$ and $d_3.$ From semigroup theory the formal solution is
\begin{eqnarray*}
\textbf{w}(\textbf{x}, t) &=& \exp\left(t D\Delta\right) \textbf{w}(\textbf{x},0) +  \int_0^t \exp\left\{ (t-\tau)\Delta\right\}\textbf{f}(\textbf{x},\tau)d\tau,
\end{eqnarray*}
where $\exp\{\}$ is the evolution operator associated with $D$. A suitable quadrature is used to approximate the integral.  Here, a second order trapezoidal rule is used, that is,
\begin{eqnarray*}
\textbf{w}(\textbf{x}, t) &=& \exp\left(t D\Delta\right) \left(\textbf{w}(\textbf{x},0) + \frac{t}{2} \textbf{f}(\textbf{x},0)\right) +  \frac{t}{2} \textbf{f}(\textbf{x},t) + \bigo{\delta t^2}.
\end{eqnarray*}
This motivates the implicit method,
\begin{eqnarray*}
\textbf{w}(\textbf{x}, t_{k+1}) &=& \exp\left(t D\Delta\right) \left(\textbf{w}(\textbf{x},t_k) + \frac{\delta t}{2} \textbf{f}(\textbf{x},t_k)\right) \\
&& ~ +  \frac{\delta t}{2} \textbf{f}(\textbf{x},t_{k+1}) + \bigo{\delta t^2},
\end{eqnarray*}
where $t_{k+1} = t_k + \delta t$.   The exponential is approximated through a Peaceman-Rachford operator.  This creates the second order implicit method,
\begin{eqnarray*}
\textbf{w}(\textbf{x}, t_{k+1}) &=& (I - \frac{\delta t_k}{2} B)^{-1} (I-\frac{\delta t_k}{2}A)^{-1}(I+\frac{\delta t_k}{2}A)(I+\frac{\delta t_k}{2}B) \left(\textbf{w}(\textbf{x},t_k) \right. \\
&& ~~ \left. + \frac{\delta t_k}{2} \textbf{f}(\textbf{x},t_k)\right) +  \frac{\delta t_k}{2} \textbf{f}(\textbf{x},t_{k+1}) + \bigo{\delta t_k^2}, \nonumber
\end{eqnarray*}
\noindent where $\delta t_k$ is the variable time step, $A = \frac{\partial^2}{\partial x^2}$, and $B=\frac{\partial^2}{\partial y^2}$.  The last reactive term does require the solution at the next time step.  To avoid a required nonlinear solve, we use a first order approximation to avoid this.  This simplification maintains the second order accuracy of the approximation.  We write the solution in an alternative form,
\begin{eqnarray*}
(I - \frac{\delta t_k}{2} A) (I-\frac{\delta t_k}{2}B)\textbf{w}(\textbf{x}, t_{k+1}) &=& (I+\frac{\delta t_k}{2}A)(I+\frac{\delta t_k}{2}B) \left(\textbf{w}(\textbf{x},t_k) + \frac{\delta t_k}{2} \textbf{f}(\textbf{x},t_k)\right) \\
    && +  \frac{\delta t_k}{2} (I - \frac{\delta t_k}{2} A) (I-\frac{\delta t_k}{2}B)\textbf{f}(\textbf{x},t_{k+1}) + \bigo{\delta t_k^2}, \nonumber
\end{eqnarray*}
This can be conveniently solved through ADI procedures \cite{Strikwerda}, that is, by splitting the problem into,
\begin{eqnarray*}
(I - \frac{\delta t_k}{2} A) \tilde{\textbf{w}}(\textbf{x}, t_{k+1}) &=& (I+\frac{\delta t_k}{2}B) \textbf{w}(\textbf{x},t_k) + \frac{\delta t_k}{2} \textbf{f}(\textbf{x},t_k) \\
(I - \frac{\delta t_k}{2} B) \tilde{\textbf{w}}(\textbf{x}, t_{k+1}) &=& (I+\frac{\delta t_k}{2}A) \tilde{\textbf{w}}(\textbf{x},t_k) + \frac{\delta t_k}{2} \textbf{f}(\textbf{x},t_{k+1}).
\end{eqnarray*}
We see that the first equation keeps $A$ implicit while $B$ is explicit.  We then take our intermediate solution, $\tilde{\textbf{w}}$, and solve the second equation keeping $B$ implicit and $B$ explicit.  Now, the spatial operators may be approximated through the two-dimensional Chebyshev approximation similar to that of the previous section, however we choose second order central differences.  Let $x_i = -1 + i h_x$ and $y_j = -1 + j h_y$, where $h_x = 2/(N-1)$, $h_y=2/(M-1)$, for $i=0, \ldots, N-1$, and $j=0, \ldots, M-1.$  Let $A_{h}$ and $B_{h}$ be the second order approximations to the operators $A$ and $B$. The approximation is utilized throughout the entire two-dimensional domain.  At the boundary, we require an equation for $\textbf{w}$ at \textit{ghost} points, that is locations of $x_{-1},~ y_{-1},~ x_{M+1}$ and $y_{M+1}$. These are established using central difference approximations to the Neumann boundary conditions.  We then substitute these back into the system of equations. This maintains the second order accuracy and the tridiagonal structure of the equations. At each step, the tridiagonal equations in the ADI method are directly solved using the Thomas algorithm which comes at a expense of $\bigo{NM}$.

\subsection{Numerical Results}\label{sec:numresults}
We numerically demonstrate that there is good evidence for conjectures \eqref{c1}, \eqref{c2}, and \eqref{c3}. In the one dimensional setting, we translate and scale the domain to $(0,\pi)$ and use a refuge function of
$$b_1(x) = \frac{1-\tanh\left(\frac{x-a}{.04}\right)}{2}$$
The refuge size is delineated by the value $a$ since that is the location where the gradient is steepest. As expected the blow up time is affected by the presence of a refuge. For instance, for a fixed parameter set we look at the effect on the blow up time with a refuge, refuge with role reversal, and role reversal and overcrowding as compared to the classical model.  The results are shown in Figure \ref{NumBlowupvsRefuge}.

 \begin{figure}[h]
 \centering{
  {\includegraphics[scale=.45]{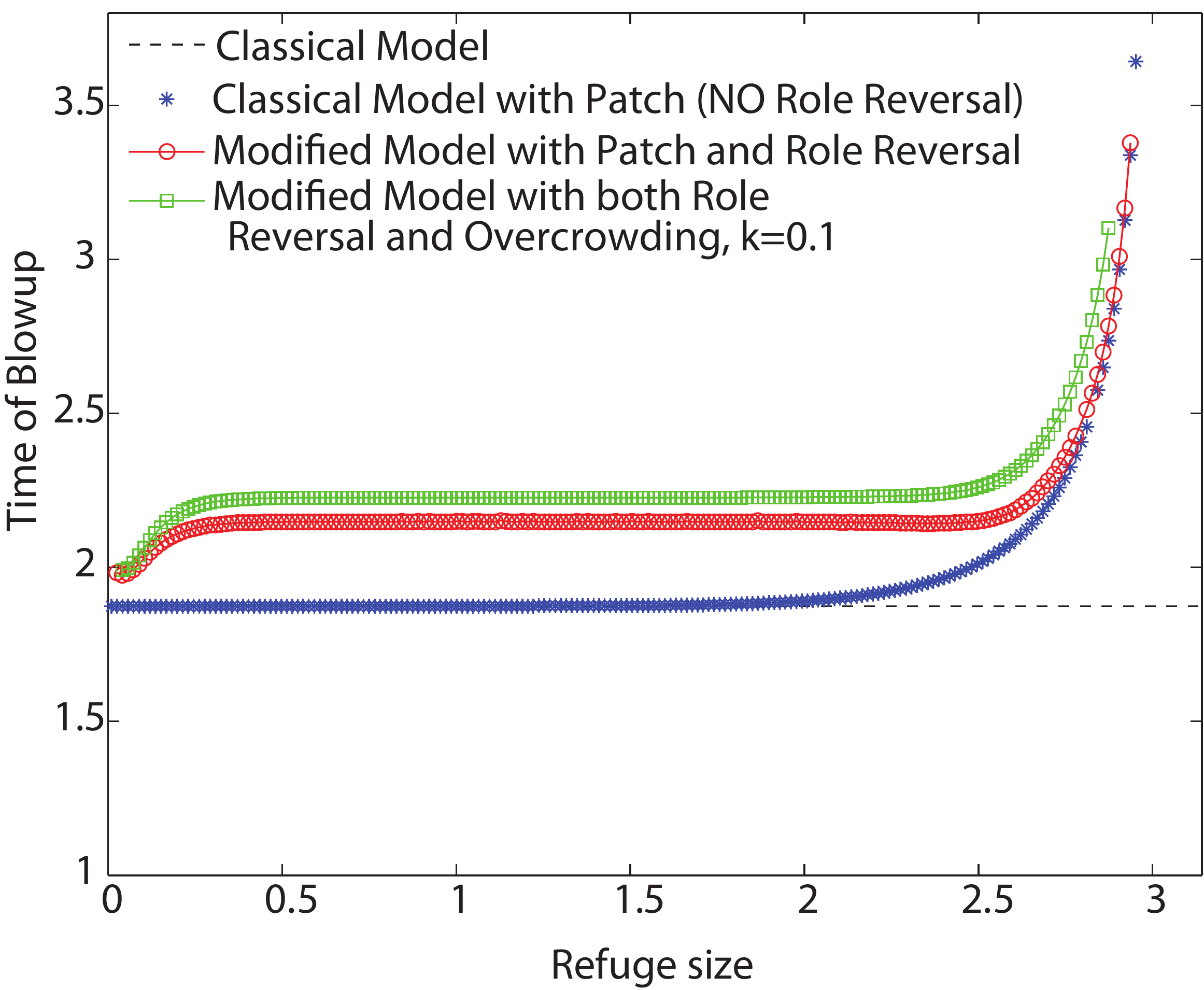}}
  }
 \caption{\small This shows blow up times for the classical model ($b_1(x)=0$ and $d_{4}=0$) versus the modified model with the various biological controls. One sees that there is a critical patch size, such that for patches of length greater than this, there is no blow up. Inclusion of the overcrowding term, greatly decreases this critical size. The parameters used are $a_1 = 1,~ a_2 = 1,~ b_2=0.5;
  D_0 = 10,~ D_1 = 13,~ D_2 = 10,~ D_3 = 20,~ D_4=D_2,~
   c = 0.055,~ w_0 = 0.55,~ w_1 =0.1,~ w_2 = 0.25,~  w_3 = 1.2,~ w_4=100,~ w_5=0.55,~ d_3 = d_2 = d_1 = .1,~
  d_4=k|c-\frac{w_3}{D_3}|$, where $k$ is given in the plot. The initial conditions used are $u(x,0)=r(x,0)=10$ and $v(x,0)=2000.$ 128 grid points were used with a temporal step size of $10^{-3}.$ In the case of overcrowding and role reversal, as we increase the value of $k$ similar results were observed.}
 \label{NumBlowupvsRefuge}
\end{figure}

Interestingly, in the modified model with role reversal, it is found that blow up does not occur when $b_1(x)=0$ or $b_1(x)=1, \forall x$ in the spatial domain.

We also see that the blow up times are influenced by the size of the refuge. In fact, there is critical refuge size for which given any refuge size greater then the solutions will not blow up. This is evidenced in Figure \ref{NumBlowupvsRefuge} by the steep gradient of the curves around $2.8$.  Therefore, the population of the invasive species can be controlled. In situations where there is only a spatial refuge the blow up time curve is always increasing.  This is consistent with our two dimensional results.

It is clear that there exists a critical refuge size for which refuges larger than this will prevent blow up.  In the one-dimensional model, without role reversal or overcrowding effects, we investigate the effect of the critical refuge size versus the size of the initial condition of $v$ while maintaining the other initial conditions and parameters values.  We use the same parameters as given in Figure \ref{NumBlowupvsRefuge} and vary the uniform initial condition on $v$.  Interestingly, Figure \ref{Num:RefugevsIC} shows a logarithmic dependency of the critical refuge size to prevent blow-up of $r$, on the initial condition size of $v$.

 \begin{figure}[h]
\begin{center}
  \includegraphics[scale=.4]{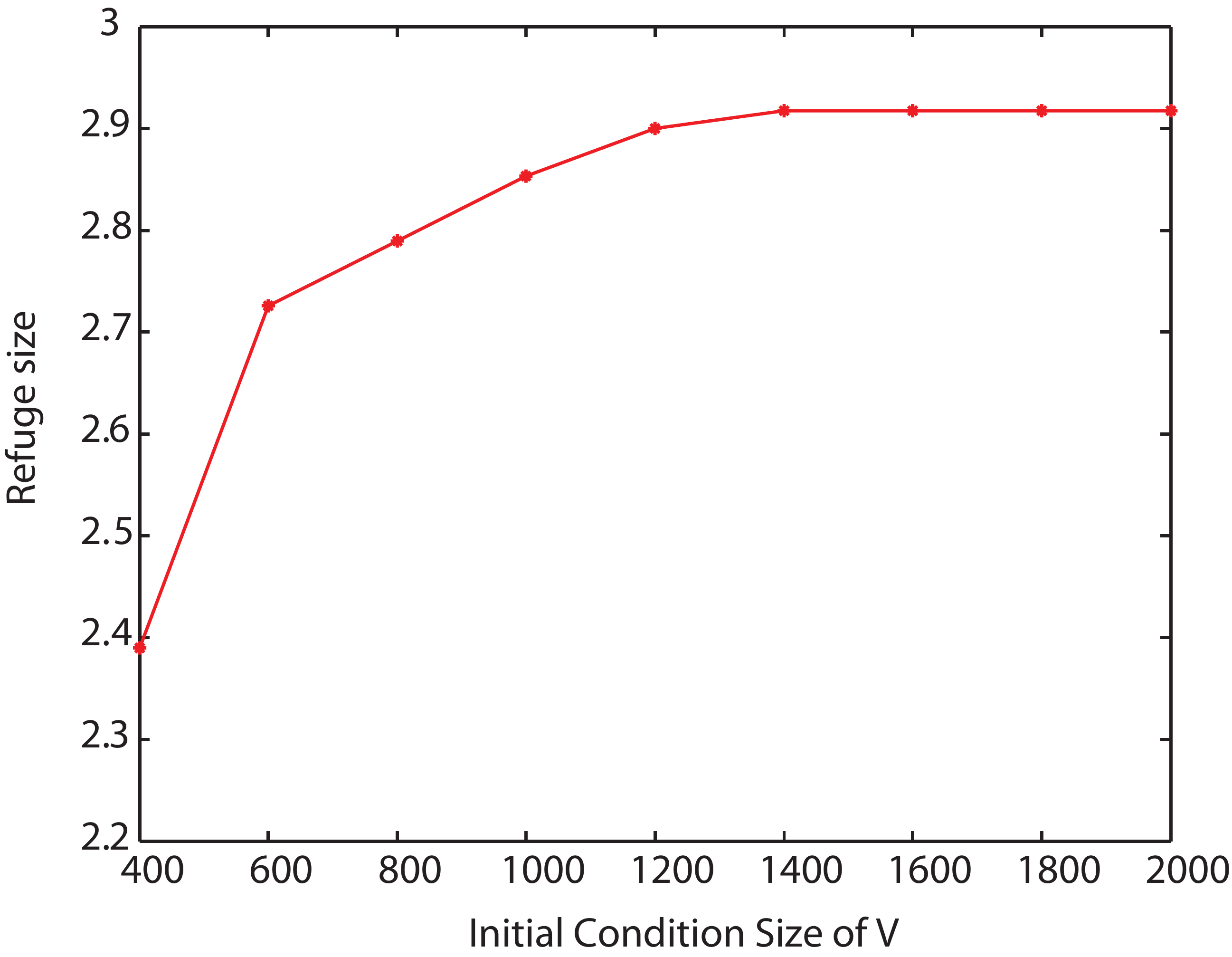}
  \caption{\small This shows a plot of the critical patch size for a given initial uniform condition size of $v$ for which blow up in the invasive specie population $r$ will occur.  Identical parameters and resolution are used as in Figure \ref{NumBlowupvsRefuge}. The domain was translated and scaled to $(0,\pi)$.}
  \label{Num:RefugevsIC}
\end{center}
 \end{figure}

The size of the spatial refuge has an influence on the blow up time.  However, the location and number of spatial refuges also influences the blow up time.  In fact, blow up is sometimes eliminated depending on the initial condition, refuges, and their subsequent locations.  For instance, if we consider a uniform initial condition in $r$ and compare the blow up time in a situation of a single refuge of width located near the boundary versus the case of evenly splitting this refuge then we find the blow up time is increased.  This delay is exasperated if the two refuges are farther apart.
Of course, if we consider a different initial condition then blow may not just be delayed, rather removed!  For instance if we consider a concentrated initial population of $r$ for which the highest concentration of $r$ contains the spatial refuge then blow up can be eliminated! Hence, the concentration of $r$ inside the refuge may protect the species enough so that the population of $r$ decays sufficiently to avoid blow up in its population.

\begin{figure}[ht]
\begin{center}
\includegraphics[scale=.4]{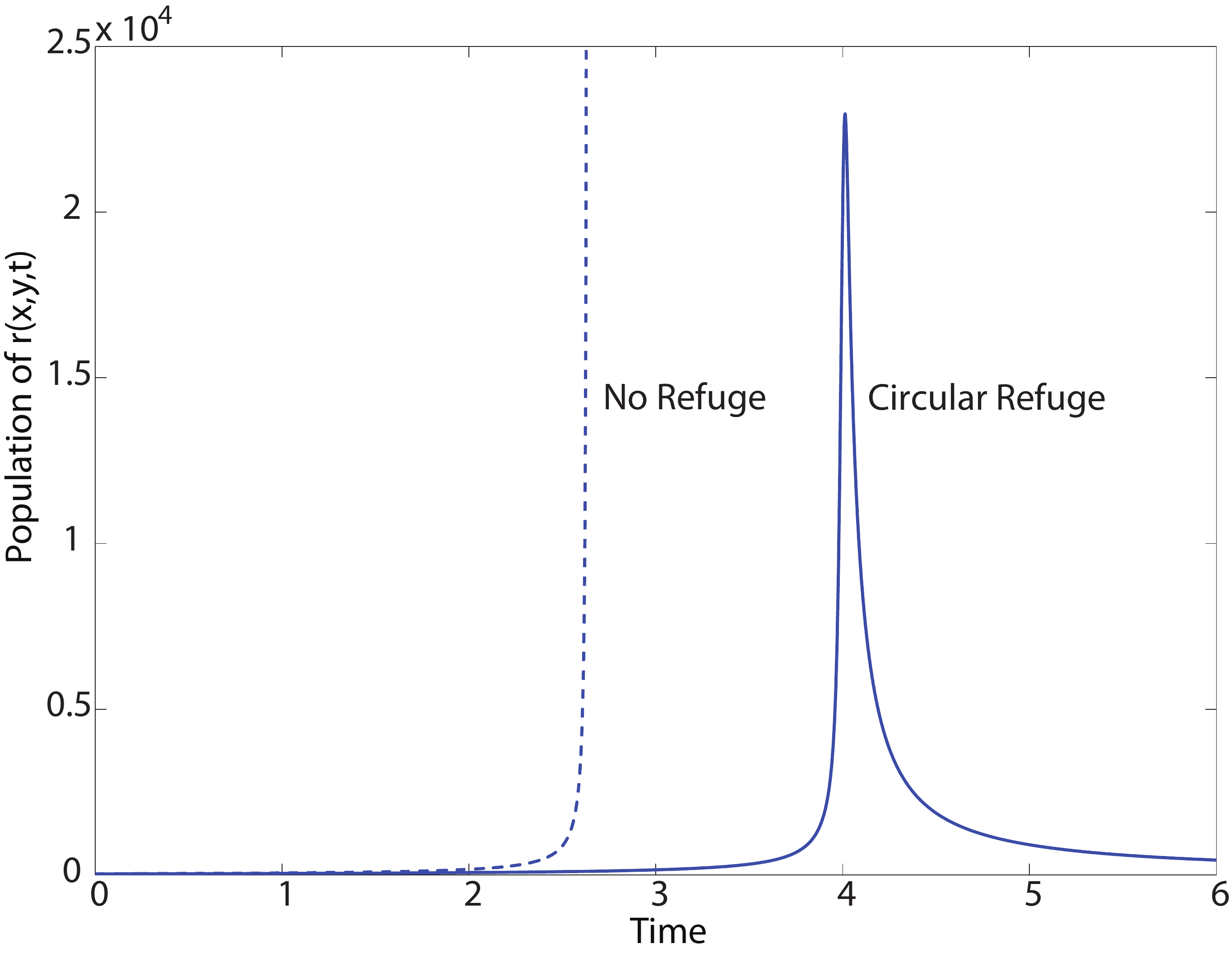}
\caption{The population of $r(x,y,t),$ that is, $\int_{\Omega} r(x,y,t) d\Omega$, is shown versus time.  With a circular refuge, we see that blow up is avoided. This provides experimental evidence that the location and size of the refuge is important to avoid blow up in the invasive population. The dashed line represents the population in the case where there is no spatial refuge. The simulations were done on a $50\times 50$ uniform grid with a temporal step of $.0001$.}
\label{rpopulationavoided}
\end{center}
\end{figure}

To illustrate this consider the two-dimensional model \eqref{eq:x3pa}, \eqref{eq:x2pa}, and \eqref{eq:(1.1)} with an initial condition of
\begin{eqnarray*}
 u(x,y,0) &=&  \cos(2\pi x)\cos(2 \pi y) + 30, \\
 v(x,y,0) &=&  u(x,y,0) + 200, \\
 r(x,y,0) &=& 100 \exp{(-10(x^2+y^2))},
\end{eqnarray*}
\noindent with parameters $d_1=d_2=d_3=.1$, $a_1=5$, $a_2=.75$, $b_2=.5$, $w_0=.55$, $w_1=1$, $w_2=.25$, $w_3=1.2$, $c=.055$, $D_0=20$, $D_1=13$, $D_2=10$, $D_3=20$. Clearly, the coefficient $c-\frac{w_3}{D_3}<0$. If $b_1(x,y)$ is zero throughout the entire spatial domain there exists blow up in the $r$ population.  However, if we have a circular refuge such that,
$$ b_1(x,y) = \left\{ \begin{array}{ll} 1 & x^2+y^2 < R \\ 0 & \mbox{else}\end{array}\right.,$$
\noindent for $R=.5$ then blow up is avoided. Figure \ref{rpopulationavoided} shows the total population versus time.  We can see the population starts to increase rapidly, but the increase is attenuated as a result of the spatial refuge.  Hence, the population growth is not sustained and begins to decrease.  Hence, the size \textbf{and} location of the refuge and the initial condition play a delicate balance in preventing blow up.

In two dimensions if we choose a particular shape of the spatial refuge it was conjectured that there exists a critical size for which blow up is avoided.  Here, we look at two situations: a square and circular refuges with increasing size centered in the middle of the spatial domain. We let $b_1(\x)=1$ inside the refuge while $0$ outside.  This clearly delineates the protection zones.  For a fixed parameter regime and initial conditions of
\begin{eqnarray*}
&& u(x,y,0) = r(x,y,0) = \cos(2\pi x)\cos(2 \pi y) + 30, \\
&& v(x,y,0) = \cos(2\pi x)\cos(2 \pi y) + 230,
\end{eqnarray*}
\noindent we determine the critical refuge size.  Each calculation is carried out with a $50\times 50$ grid.  The temporal step was fixed at $.001$.  The parameters used: $d_1=d_2=d_3=.1$, $a_1=1$, $b_1=.5$, $w_0=.55$, $w_1=2$, $w_2=.25$, $w_3=1.2$, $c=.055$, $D_0=20$, $D_1=13$, $D_2=10$, $D_3=20$.  It is found that for a square the critical refuge area is $2.5004,$ roughly $62.5\%$ of the spatial domain.  The critical refuge area for the circular refuge is approximately $2.6661$, roughly $66.7\%$ of the spatial domain.

\section{Spatio-temporal Dynamics}
\label{6}

\subsection{Turing Instability: Effect of Overcrowding on Turing Patterns}

In this section we shall investigate the effects of overcrowding in the absence of refuges and role reversal, in the classical model. Therefore, we focus on whether an appropriate choice of $d_4$ can induce Turing Instabilities.  It is shown in \cite{PK14} that diffusion processes can destabilize the homogenous steady state solution when $d_4=0$. For convenience, we restate the one-dimensional model given in equations \eqref{eq:(1.1)}, \eqref{eq:x3pc}, and \eqref{eq:x2pam2},
\begin{eqnarray}
\frac{\partial u}{\partial t}&=&d_1 u_{xx} + a_1u - b_2u^2 - w_0\bigg({uv\over {u+D_0}}\bigg), \label{eq:1.1a}\\
\frac{\partial v}{\partial t}&=&d_2 v_{xx} - a_2v + w_1\bigg({uv\over {u+D_1}}\bigg) - w_2\bigg({vr\over {v+D_2}}\bigg),\label{eq:1.1b}\\
\frac{\partial r}{\partial t}&=&d_3 r_{xx} + cr^2 - w_3 {r^2 \over {v +D_3}} + d_4(r^2)_{xx}.\label{eq:1.1c}
\end{eqnarray}
Consider the linearization of \eqref{eq:1.1a}-\eqref{eq:1.1c} about the positive interior equilibrium point
\[
E_6=\bigg(u^* ,v^*,r^* \bigg),
\]
\noindent where $u^*=\frac{a_1-b_2 D_0}{2 b_2} + \sqrt{ (\frac{a_1-b_2 D_0}{2 b_2} )^2 - ( \frac{w_0 v^* - a_1 D_0}{b_2} ) }$, $v^*$ is the the spatially homogenous steady state solution, and $r^*=\frac{v^* + D_2}{w_2} (\frac{w_1u^*}{u^* + D_1} - a_2)$ (see \cite{PK14}).  For instance, for the parameters given at the beginning of this section $u^{*}=10.110031,~ v^{*}=10,$ and $r^{*}=2.997897$. Consider a small space time perturbation, that is,
\[ {\bf{W}}={U}- {U^*}=O(\epsilon),\,\text{where}\quad \epsilon \rightarrow 0, \]
\noindent with  $U^*$ as the positive interior equilibrium point given as $E_6$ and $U=(u,v,r)$. Substituting and collecting linear terms of order $\bigo{\bf W}$, we obtain
\begin{eqnarray*}
{\partial {\bf W } \over \partial t} &=& {\bf D}\Delta {\bf W} + {\bf J W}, \\
\Delta {\bf W}_i \cdot{\bf n} &=& 0 \quad \text{for} \: x\in \partial \Omega, i=1,2,3.
\end{eqnarray*}
\noindent where
\[
{\bf D_\textit{u*,v*,r*}}=\begin{bmatrix}
       d_1 & 0 & 0\\
       0 & d_2 & 0\\
       0 &  0    & d_3 + 2d_4r^*
       \end{bmatrix},
\]
\noindent is the diffusion matrix and
\[
{\bf J_\textit{u*,v*,r*}}=\begin{bmatrix}
 u^*\bigg(-b_2 + {w_0v^*\over {(u^* +D_0)^2}}\bigg)  & -{u^*w_0\over u^* + D_0}  & 0\\
  {v^*D_1w_1 \over (u^* + D_1)^2}                             &  { v^*w_2r^*\over (v^* + D_2)^2}           &-{v^*w_2\over {v^* + D_2}}\\
       0 & { {r^*}^2w_3\over (v^* + D_3)^2}    & 0
       \end{bmatrix}=
       \begin{bmatrix}
       A_{11} & A_{12} & A_{13}\\
       A_{21} & A_{22} & A_{23}\\
       A_{31} & A_{32} & A_{33}
       \end{bmatrix},
\]
\noindent is the Jacobian matrix associated with the ordinary differential equation part of model \eqref{eq:1.1a}-\eqref{eq:1.1c}.

Let
$$ {\bf W}(\varepsilon,t)=
 \begin{bmatrix}
     \bar{u}_0\\
      \bar{u}_1\\
       \bar{u}_2
       \end{bmatrix}e^{\lambda t + i{ k}\varepsilon},
$$
\noindent where $ \varepsilon$ is the spatial coordinate in $\Omega$, $\bar{u}_i (i=0,1,2)$ is the amplitude, $\lambda$ is the eigenvalues associated with the interior equilibrium point, $E_6$ and $k$ is the wave number of the solution.  Upon substituting, we obtain the characteristic equation
\begin{align}\label{equ:1.2}
|{\bf J} - \lambda {\bf I} -  k^2{\bf D} |=0,
\end{align}
where ${\bf I}$ is a $3\times3$ identity matrix. The sign of $Re(\lambda)$ indicates the stability, or lack thereof, of the equilibrium point $E_6$. The dispersion relation is
\[
P(\lambda) = A_3(k^2)\lambda^3 + A_2(k^2)\lambda^2 + A_1(k^2)\lambda + A_0(k^2).
\]
The coefficients of $P(\lambda)$ are determined by expanding \eqref{equ:1.2}, namely,
\begin{align}
\begin{split}
A_3(k^2) =& 1,\\
A_2(k^2) =& \big(d_1+d_2+\left(d_3 + 2d_4r^*\right)\big)k^2 - A_{11} - A_{22} - A_{33}, \\
A_1(k^2) =& A_{22}A_{33} - A_{22}(d_3 + 2d_4r^*)k^2 - A_{23}A_{32} - A_{12}A_{21} - d_2A_{33}k^2 \\
&+ d_2(d_3 + 2d_4r^*)k^4 + \bigg(   \left( d_1k^2 - A_{11}\right)\left( d_2k^2 + (d_3 + 2d_4r^*)k^2 \right. \\
&\left.- A_{22} -A_{33}\right)\bigg), \\
A_0(k^2) = &\bigg( \left(  d_1k^2 - A_{11}\right)\big( A_{22}A_{33} - A_{22}(d_3 + 2d_4r^*)k^2 - A_{23}A_{32} - A_{33}d_2k^2 \\
&+ d_2(d_3 + 2d_4r^*)k^4 \big) \bigg) - A_{12}A_{21}(d_3 + 2d_4r^*)k^2 + A_{12}A_{21}A_{33}. \nonumber
\end{split}
\end{align}
To check for stability of the equilibrium solution $E_6$, we use the Routh Hurwitz criterion. This state that for $E_6$ to be stable we need
\begin{align}\label{equ:1.3}
A_n(k^2)>0,~ \forall n \quad \text{and}\quad A_1(k^2)A_2(k^2) >A_0(k^2).
\end{align}
Contradicting either of these statements ensures instability for $E_6$. Finally, for diffusion to cause a Turing instability it is sufficient to require that around the equilibrium point we have
\begin{eqnarray*}
\mathbb{R}e(\lambda(k=0))&<&0,~ \:\text{and}\\
\mathbb{R}e(\lambda(k>0))&>&0.
\end{eqnarray*}
We refer the reader to a well detailed analysis of this in \cite{Gilligan}. Hence, for a Turing instability to occur, we require that \eqref{equ:1.3} is satisfied when $k=0$ (without diffusion) and at least one of the equations in \eqref{equ:1.3} changes sign when $k>0$ (with diffusion). By this we consider letting $A_0(k^2)$ become negative when $k>0$ and satisfy \eqref{equ:1.3} when $k=0$.  This suggests that spatial patterns should be observed.
Our parameter search has not yielded a parameter set for which \eqref{equ:1.3} holds while $d_4=0.$ When $d_{4} > 0$, at least one of these inequalities changes sign. Thus we cannot conclusively say that $d_{4}$ will induce or inhibit Turing pattern. However, it is conclusive that $d_{4}$ certainly has an effect on the type of patterns that do form. In particular, the patterns fall into two types: spatial patterns and spatio-temporal patterns. The conditions for which type forms is succinctly described via Table \ref{table:pa}.

\begin{table}[h]
\begin{center}
\begin{tabular}{c|c|c|c|c} \hline
Case & $A_0$ & $A_2A_1-A_0$ & $A_1$ & Pattern Type \\ \hline
1    & +     & --           & --    & spatio-temporal \\ \hline
2    & --    & +            & +     & fixed spatial \\ \hline
\end{tabular}
\caption{Conditions on the signs of coefficients for different types of patterns.}
\label{table:pa}
\end{center}
\end{table}

Now when $d_{4} > 0$ this clearly changes the spatio-temporal pattern, as evidenced by the sign changes from Case 1 to 2 in Table \ref{table:pa}.  In this situation the dispersion relation is shown in Figure \ref{DispersionPlot2}.  The resulting patterns are shown Figure \ref{SpatioTemporalPatternsUVR}(a)-(f).
\begin{figure}[h]
  \centering
    \includegraphics[scale=.4]{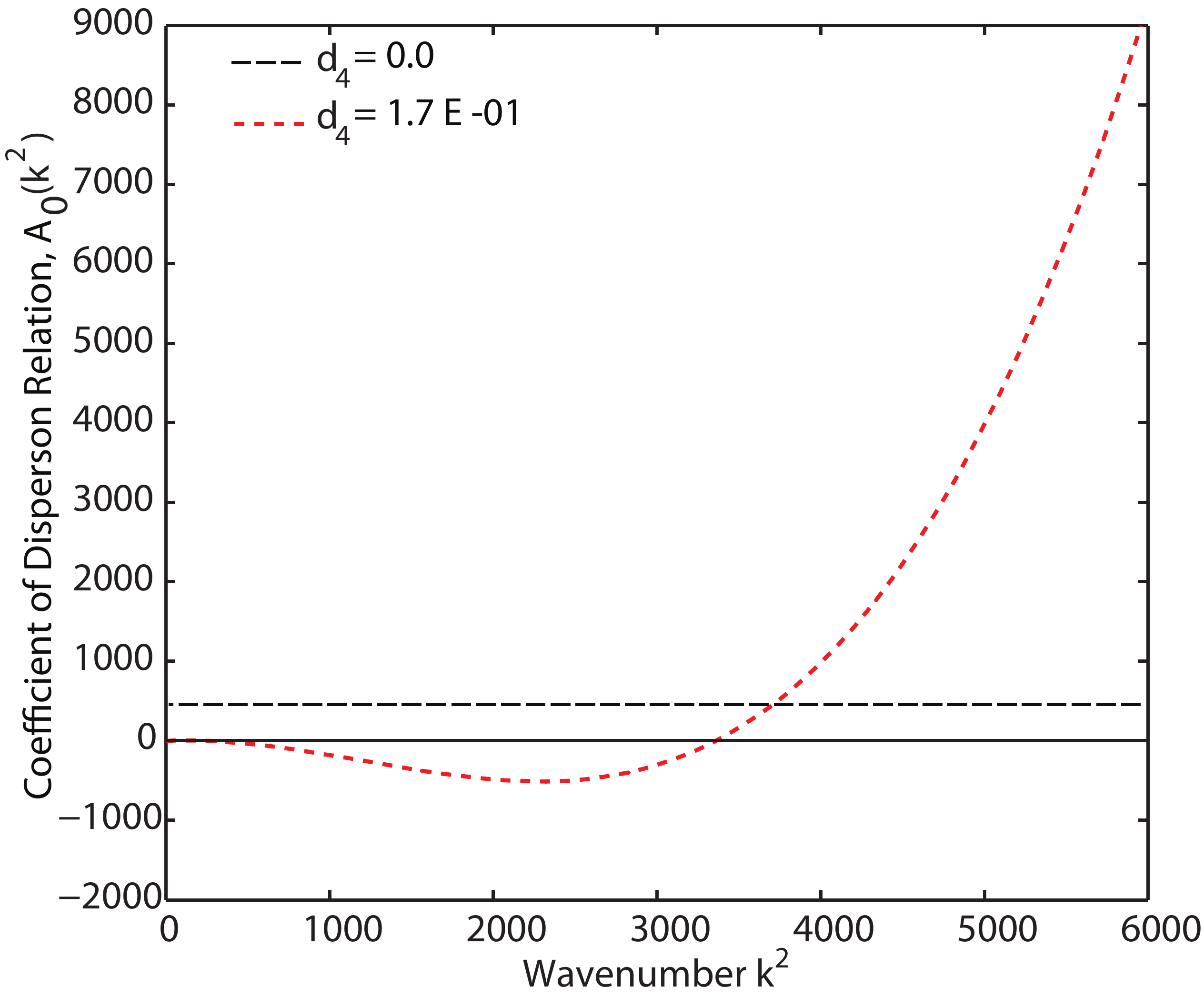}
      \caption{Dispersion Plot for two choices of $d_4$ resembling the two cases in Table \ref{table:pa}.  Notice that when $d_4>0$ a band of wavenumbers are unstable.}
      \label{DispersionPlot2}
\end{figure}

\begin{figure}[htb!]
\begin{center}
\includegraphics[scale=.6]{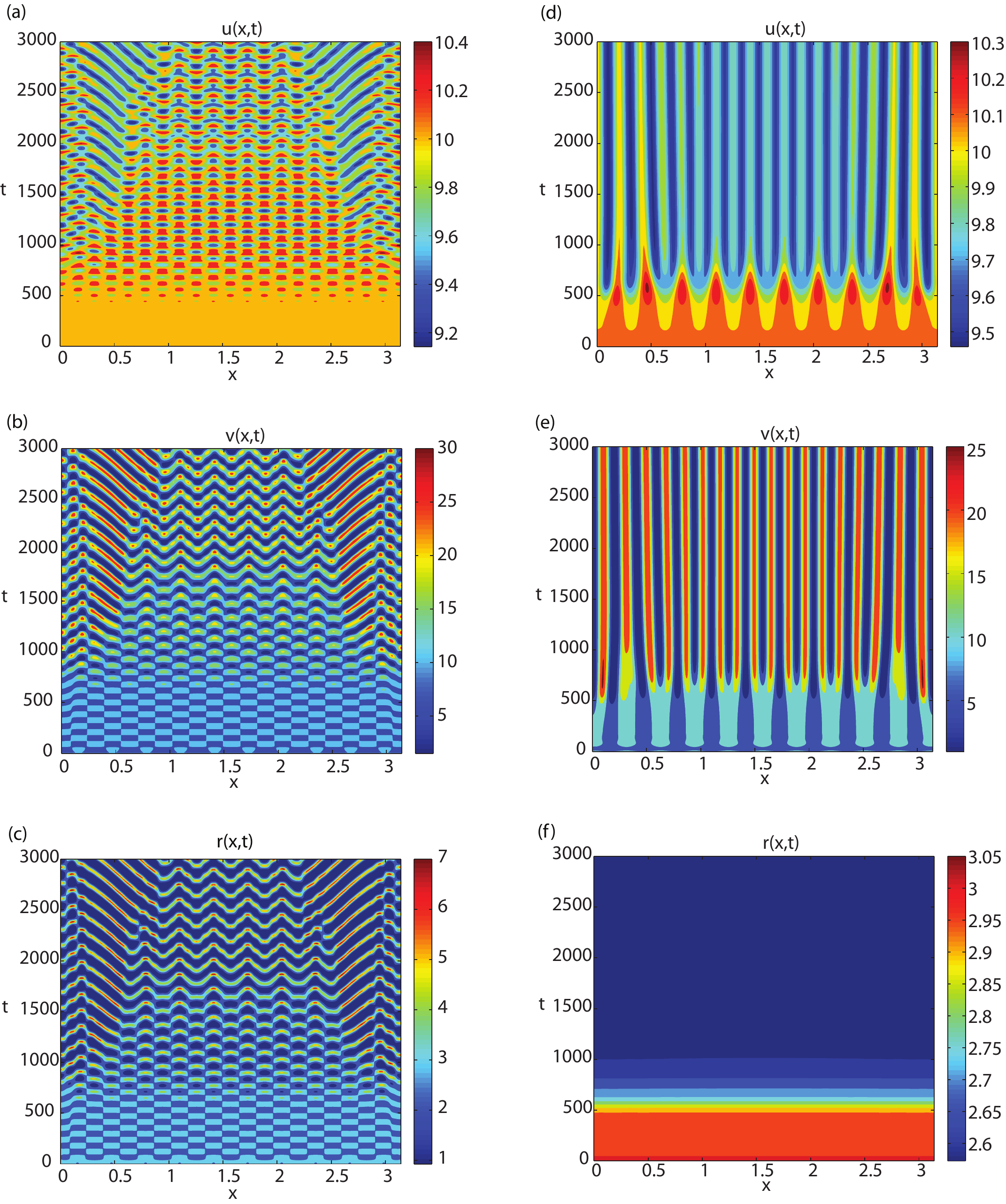}
\caption{The densities of the three species are shown as contour plots in the $xt$-plane. The long-time simulation yields (a)-(c) Turing patterns ($d_4=0$), that are spatio-temporal, and (d)-(f) stripe Turing patterns ($d_4=1.6\times 10^{-1}$), which are purely spatial.  The other parameters are: $d_1=10^{-2},~d_2=10^{-5},~d_3=10^{-7}$, $a_1=1.79,~a_2=0.8,~b_2=0.15,~c=0.04,$ $w_0=0.55,~ w_1=2,~ w_2=0.5,~ w_3=1.2,$ $D_0=10,~ D_1=13,~ D_2=10,~ D_3=20.$
$128$ grid points are used with a temporal step size of $.01$.}
\label{SpatioTemporalPatternsUVR}
\end{center}
\end{figure}

\begin{remark}
If $r=0$ then \eqref{eq:1.1a}-\eqref{eq:1.1b} reduces to a model similar to the classical predator-prey model with a Holling type II functional response, for which we know there cannot occur Turing instability.  There is one caveat, the death rate in $v$ now becomes nonlinear. So essentially the equations are
\begin{eqnarray}
\label{eq:(1.2od)}
\partial _{t}v &=& d_{2} v_{xx}-\left (a_{2} + w_{2}\frac{r^{*}}{v+D_{2}}\right)v+w_{1}\frac{uv}{u+D_{1}}, \\
\label{eq:(1.1od)}
\partial _{t}u &=& d_{1} u_{xx}+a_{1}u-b_{2}u^{2}-w_{0}\frac{uv}{u+D_{0}}.
\end{eqnarray}
\noindent Such systems have been investigated with and without cross and self diffusion \cite{xie12}.  In the regular diffusion case, it is known that a Turing instability can exist \cite{S09}, where the nonlinearity in death rate is due to cannibalism. However, to our knowledge, for the specific form of the death rate as above this is not yet known.

Thus we see that the classical model \eqref{eq:(1.3)}-\eqref{eq:(1.1)} can exhibit spatio-temporal patterns, apart from just the spatial patterns that were uncovered in \cite{PK14}. Furthermore addition of the overcrowding term in \eqref{eq:(1.3)}-\eqref{eq:(1.1)}, can cause the spatio-temporal patterns to change into a purely spatial patterns. It is noted that the inclusion of overcrowding with a nonlinear death rate, such as in \eqref{eq:(1.2od)}, can lead to Turing instability in the classical predator-prey model with Holling type II response.

\end{remark}

\subsection{Spatio-temporal chaos}

The goal of this section is to investigate spatio-temporal chaos in the classical model \eqref{eq:(1.3)}-\eqref{eq:(1.1)}. Spatio-temporal chaos is usually defined as deterministic dynamics in spatially extended systems that are characterized by an apparent randomness in space and time \cite{Cai01}. There is a large literature on spatio-temporal chaos in PDE, in particular there has been a recent interest on spatially extended systems in ecology exhibiting spatio-temporal chaos \cite{M02}. However, most of these works are on two species models, and there is not much literature in the three-species case. In \cite{PK14} we showed diffusion induced temporal chaos in  \eqref{eq:(1.3)}-\eqref{eq:(1.1)}, as well as spatial chaos when the domain is enlarged. In \cite{N13}, various patterns non-Turing were uncovered in  \eqref{eq:(1.3)}-\eqref{eq:(1.1)}, in the case of equal diffusion, but spatio-temporal chaos was not confirmed. Note, that the appearance of a jagged structure in the species density, as seen in \cite{N13}, which seems to change in time in an irregular way, does not necessarily mean that the dynamics is chaotic. One rigoros definition of chaos means sensitivity to
initial conditions. Thus two initial distribution, close together, should yield an exponentially growing difference in the species distribution at later time. In order to confirm this in \eqref{eq:(1.3)}-\eqref{eq:(1.1)}, we perform a number of tests as in \cite{M02}. We run \eqref{eq:(1.3)}-\eqref{eq:(1.1)} form a number of different initial conditions, that are the same modulo a small perturbation. The parameter set is chosen as in \cite{N13}.
We then look at the difference of the two densities, at each time step in both the $L^{\infty}$ and $L^{1}$ norms.

Thus we solve  \eqref{eq:(1.3)}-\eqref{eq:(1.1)} with the following parameter set: $d_1=10^{-5},~d_2=10^{-5},~d_3=10^{-5}$, $a_1=1.93,~a_2=1.89,~b_2=0.06,~c=0.03,$ $w_0=1,~ w_1=0.5,~ w_2=0.405,~ w_3=1,$ $D_0=10,~ D_1=10,~ D_2=10,~ D_3=20.$ Thus the steady state solution for the ODE system is $u^{*}=25,~ v^{*}=13,~ r^{*}=9$.

The simulations use two different (but close together in $L^{1}(\Omega), L^{2}(\Omega), L^{\infty}(\Omega)$ norms) initial conditions. The first simulation (which we call $r_{unpert}$) is a perturbation of $(r^{*},v^{*},u^{*})$ by $0.1\cos^{2}(x)$.  The second simulation (which we call $r_{pert}$) is a perturbation of $(r^{*},v^{*},u^{*})$ by $0.11\cos^{2}(x)$. The densities of the species are calculated up to the time $t=5000.$  At each time step in the simulation we compute
\[
d(t) = ||r_{unpert}(x,t) - r_{pert}(x,t)||_{X} ,
\]
\noindent where $X=L^{1}(\Omega),~L^{2}(\Omega) \ \mbox{ and } \ L^{\infty}(\Omega)$ are used.  Then, $d(t)$ is plotted on a log scale.  In doing son, we observe the exponential growth of the error.  This grows at an approximate rate of $0.018>0$. Since this is positive then this is an indicator of spatio-temporal chaos. These numerical tests provide experimental evidence to the presence of spatio-temporal chaos in the classical model \eqref{eq:(1.3)}-\eqref{eq:(1.1)}. Figure \ref{ContourChaos} shows the densities of the populations in the $xt$-plane while Figure \ref{ContourChaosError} gives the error and its logarithm till $t=1000$.

\begin{figure}[htb!]
	\begin{center}
    \includegraphics[scale=.5]{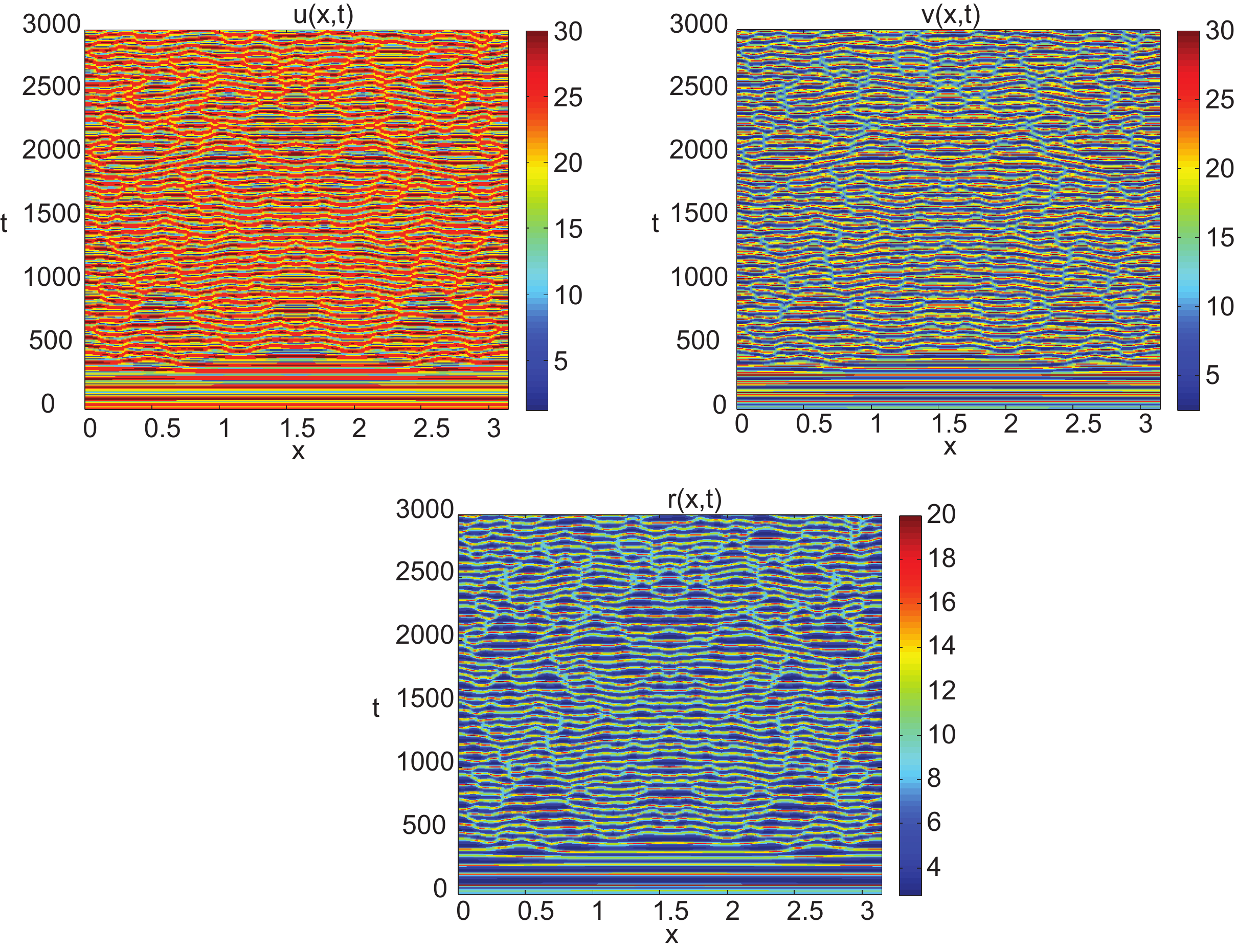}
	\end{center}
	\caption{ The densities of the three species are shown as contour plots in the $xt$-plane for $u$, $v$, and $r$ from left to right. The long-time simulation yields spatio-temporal chaotic patterns. $128$ grid points are used with a temporal step size of $.01$.}
	\label{ContourChaos}
\end{figure}
\begin{figure}[htb!]
\begin{center}
    \includegraphics[scale=0.4]{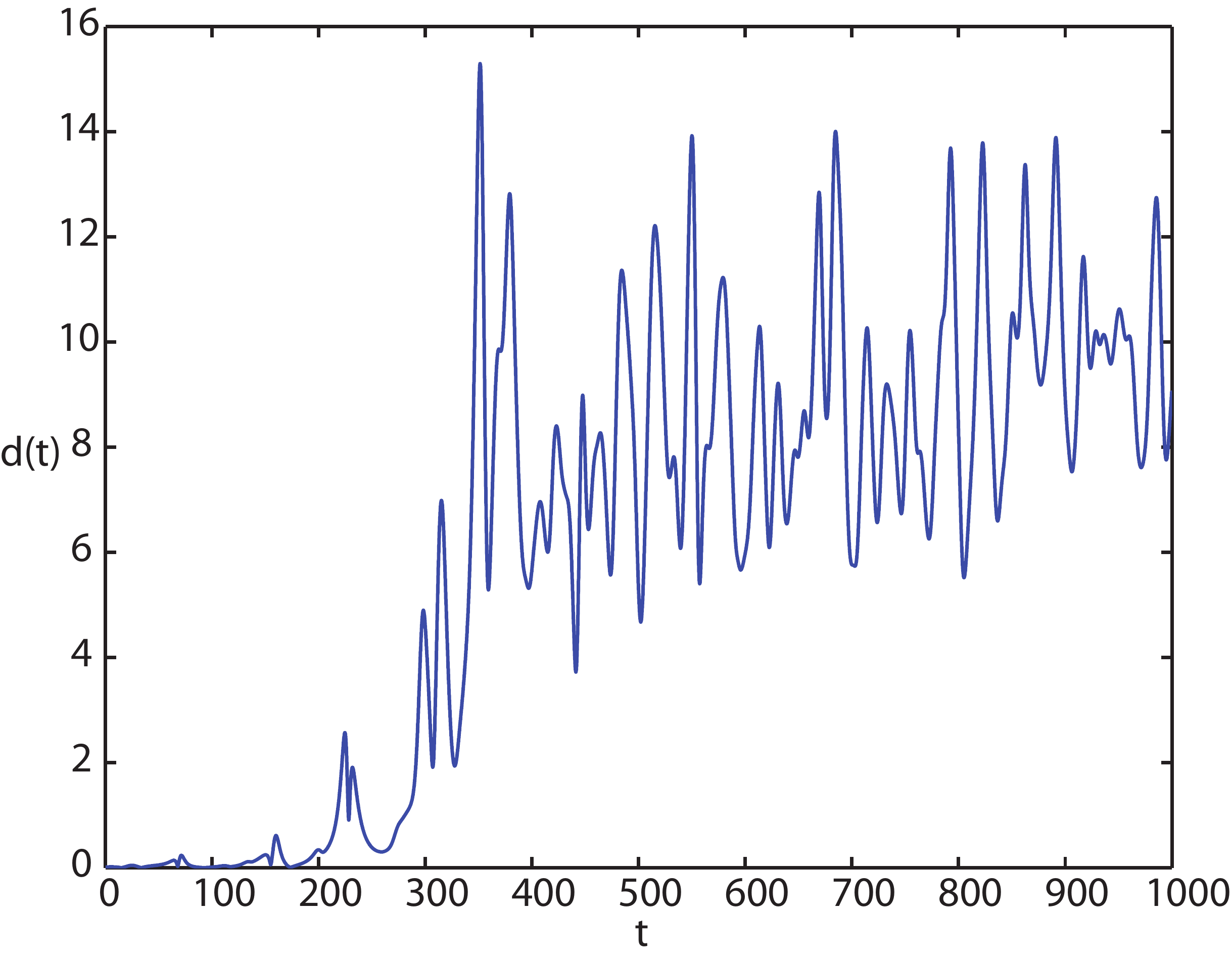}
	\includegraphics[scale=0.4]{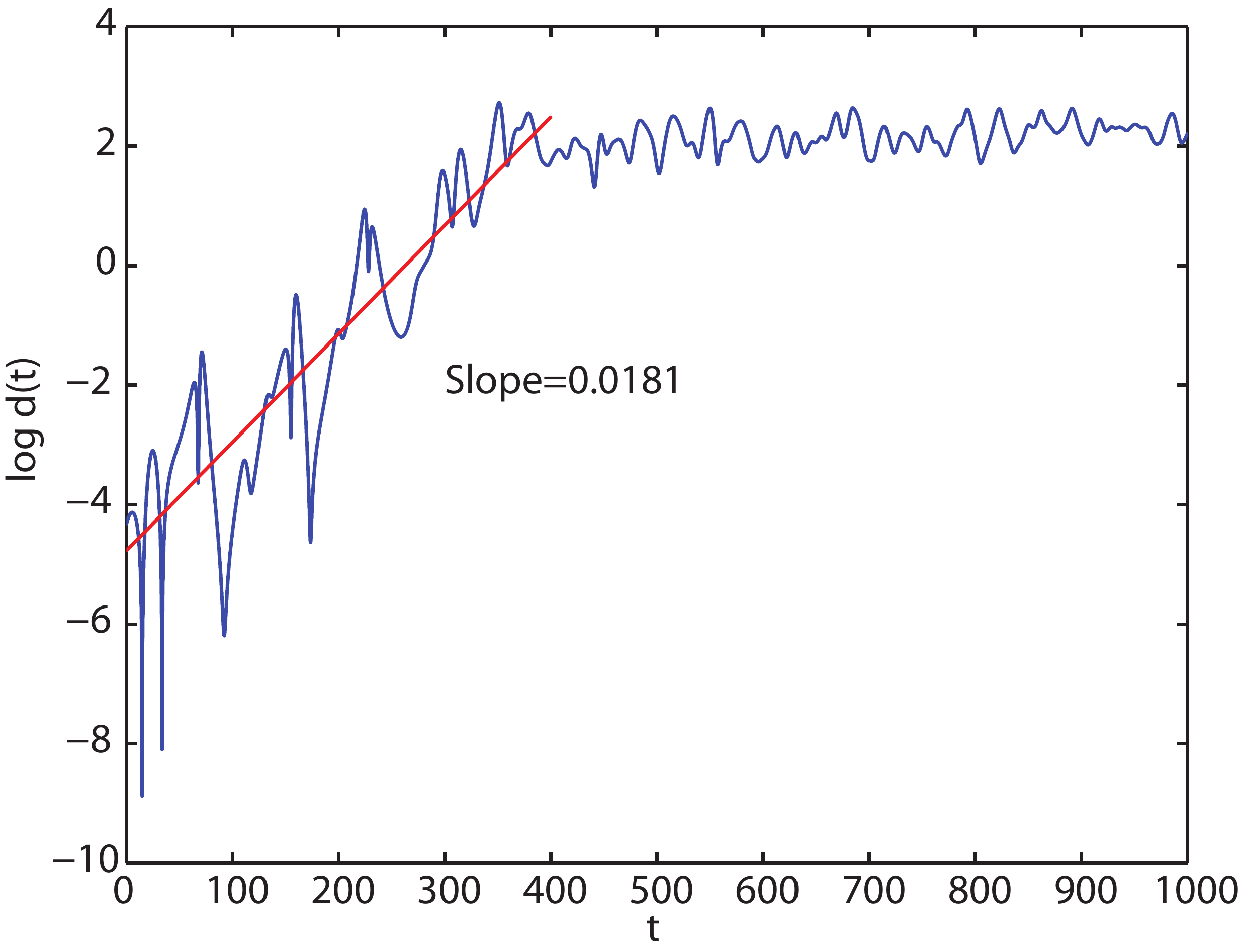}
\end{center}
		\caption{The difference $d(t)$ and its logarithm in the $L^2$ norm for the $r$ species, under slightly different initial conditions are shown. The difference is seen to grow at an exponential rate of approximately $0.018$.  Comparable results were found for the $L^{\infty}$ and $L^{1}$ norms. These tests provide experimental evidence that there presence of spatio-temporal chaos in the classical model \eqref{eq:(1.3)}-\eqref{eq:(1.1)}.}
	\label{ContourChaosError}
\end{figure}

\section{Conclusion}
\label{7}
In this work we have proposed and investigated a new model for the control of an invasive population. An invasive species population is said to reach catastrophic population levels when its population reaches a particular threshold.  The mathematical model uses the mathematical construct of finite time blow up which enables the model to examine the effect of controls for any particular threshold, especially since this level depends on the application.  In effect, this construct demonstrates that if the invasive population has large enough numbers initially, it can grow to explosive levels in finite time: thus wreaking havoc on the ecosystem. Hence, we are interested in the influence certain controls will have on the invasive population and if its population may be reduced below disastrous levels. This formulation yields a clear mathematical problem:
Assume that \eqref{eq:(1.3)}-\eqref{eq:(1.1)} blows-up in finite time for $c < \frac{w_{3}}{D_{3}}$ for a given initial condition. Are there controls and features of the model that we can include to modify \eqref{eq:(1.3)}-\eqref{eq:(1.1)} so that now \emph{there is no blow up} in the invasive population given the same initial condition? This paper addresses this question, suggesting clear controls and improvements to the mathematical model. We then investigate these improvements numerically and theoretically.

In traditional practice, biological control works on the enemy release hypothesis. That is releasing an enemy of the invasive species into the ecosystem will lead to a decrease in the invasive population. However, non-target effects are prevalent, hence this approach is problematic and may create even more devastating impacts on the ecosystem \cite{F00}. Here, we propose controls that do not use the release of biological agents. In particular, we introduce spatial refuges or safe zones for the primary food source of the invasive species. Mathematically, this transforms \eqref{eq:(1.3)}-\eqref{eq:(1.1)} into an indefinite problem \cite{G06}, that is, where the sign of the coefficient of $r^2$ switches between inside and outside of the refuge. We demonstrated numerically and in analytically, with some assumptions, that this control can prevent blow up and drive the invasive population down. Our numerical experiments suggests that there is a delicate balance between the size and location of the refuge and the initial condition.  In particular, we revealed a logarithmic dependence on the size of the initial condition for $v$ versus the critical refuge size to prevent blow-up of $r$, see Figure \ref{Num:RefugevsIC}.  Clearly, the balance is even more pronounced for multiple refuges and in higher dimensions.

We also improved the mathematical model by incorporating overcrowding effects, which also may be used a control.  We also examined the situation where a species may switch its primary food source based in regions of the domain, that is, a prey may switch to a predator, hence their roles are reversed.  This models scenarios where influences on the landscape provide a competitive advantages to certain species.  Both, role-reversal and overcrowding act as damping mechanisms and also may prevent blow-up in the invasive population.  In particular, smaller refuge sizes in conjunction with role-reversal and overcrowding are required to prevent blow up. Of course, how does one enforce overcrowding in an ecosystem such that we obtain this desired effect? Can one devise mechanisms to facilitate this dispersal of population? We suggest one approach.  Suppose we create a lure, such as a pheromone trap, that is placed in the refuge. This would lure the invasive species into the patch, where its growth would be controlled. Of course, the species would eventually exit the refuge in search for a higher concentration of food.  Hence, in the future we plan to include spatially dependent diffusion constants to model this situation.

In our mathematical models we confirmed spatio-temporal chaos. We also see that overcrowding can effect the sorts of Turing patterns that might form. Since, environmental effects are inherently stochastic, part of our future investigations introduces stochasticity into the model.  It is not known what effect this will have on the spatio-temporal chaos or Turing patterns that may emerge.

It is clear that our new mathematical modeling constructs and results are useful analytical and numerical tool for scientists interested in control of invasive species.  Moreover, the results motivate and encourage numerous avenues of future exploration, many of which are currently under study and will be presented in future papers.

\section{acknowledgement}
We would like to acknowledge very helpful conversations with professor Pavol Quittner, professor Philippe Souplet and professor Joseph Shomberg, as pertains to the analysis of indefinite parabolic problems, as well as finite time blow-up in the superlinear parabolic problem, under various boundary conditions, and initial data restrictions.

 \end{document}